\documentclass[11pt]{article}

\usepackage[a4paper,truedimen,mag=1200]{geometry}
\usepackage{}
\usepackage{amssymb}
\usepackage{amsthm}
\usepackage{amsmath}
\usepackage{ascmac}
\usepackage{amscd}
\usepackage{ytableau}
\usepackage[dvipdfmx]{graphicx}
\usepackage{setspace}
\usepackage{color}
\title{Relation between the Weyl group orbits of fundamental weights for multiply-laced finite dimensional simple Lie algebras and d-complete posets}
\date{}
\author{Masato Tada \footnote{Graduate School of Pure and Applied Sciences, University of Tsukuba, 1-1-1 Tennodai, Tsukuba, Ibaraki 305-8571, Japan \newline e-mail: t-d-masato@math.tsukuba.ac.jp}}
\begin{document}

\newcommand{\Lieg}{\mathfrak{g}}
\newcommand{\Lieh}{\mathfrak{h}}
\newcommand{\ULieg}{U(\mathfrak{g})}
\newcommand{\Field}{\mathbb{F}}
\newcommand{\Fcomp}{\mathbb{C}}
\newcommand{\Fqcomp}{\mathbb{C}(q)}
\newcommand{\Adef}{\overset{\text{def}}{\iff}}
\newcommand{\Edef}{\overset{\text{def}}{=}}

\theoremstyle{definition}
\newtheorem{prop}{Proposition}[section]
\newtheorem{axio}[prop]{Axiom}
\newtheorem{defi}[prop]{Definition}
\newtheorem{lemm}[prop]{Lemma}
\newtheorem{theo}[prop]{Theorem}
\newtheorem{coro}[prop]{Corollary}
\newtheorem{rema}[prop]{Remark}
\newtheorem{exam}[prop]{Example}
\newtheorem{theorem}{定理}
\renewcommand{\thetheorem}{\!\!}

\ytableausetup{mathmode,boxsize=2.4em}
\ytableausetup{aligntableaux=center}

\makeatletter
\renewcommand{\theequation}{%
\thesection.\arabic{equation}}
\@addtoreset{equation}{section}
\makeatother

\maketitle

\begin{abstract}
It is known that there exists an order isomorphism between the Weyl group orbit through a minuscule weight of a simply-laced finite-dimensional simple Lie algebra and the set of all order filters in a self-dual connected d-complete poset. 
In this paper, we try to extend this fact to the case of multiply-laced finite-dimensional simple Lie algebras by using the ``folding'' technique with respect to a Dynkin diagram automorphism.
\end{abstract} 


\section{Introduction.}\label{Sec:intro}

A d-complete poset, introduced by Robert A. Proctor (\cite{Proc1,Proc2}), is a finite poset that satisfies some local conditions described in terms of double-tailed diamonds (see Section \ref{Sec:d-comp}).
A d-complete poset is one of the extensions of Young diagrams or shifted Young diagrams; in fact, a d-complete poset has an extension of some Young diagram's properties such as the hook length property (\cite{Proc3}) and the jeu de taquin property (\cite{Proc4}).
So, it is expected that d-complete posets will be used in combinatorial representation theory as well as Young diagrams and shifted Young diagrams are used.

Now, we recall the fundamental relation between d-complete posets and finite-dimensional simple Lie algebras (see Section \ref{Sec:Lie-d-comp}).
Let $\Lieg$ be a simply-laced finite-dimensional simple Lie algebra, with $W=\langle s_i\mid i\in I\rangle$ the Weyl group.
Let $\lambda$ be a dominant integral weight of $\Lieg$, and set $W_\lambda:=\{w\in W\mid w\lambda=\lambda\}$.
We define an order $\le_s$ (resp., $\le_w $) on $W\lambda$ which corresponds to the Bruhat order (resp., weak Bruhat order) under the canonical map $W\lambda \overset{\sim}{\rightarrow} W/W_\lambda \subset W$.
If $\lambda$ is minuscule (in this case, $\le_s$ is identical to $\le_w $), then there exists a connected self-dual d-complete poset $(P_\lambda,\le)$ such that $(W\lambda,\le_s) = (W\lambda,\le_w)$ and $(\mathcal{F}(P_\lambda),\subseteq)$ are isomorphic as posets (\cite[Section 14]{Proc1}), where $\mathcal{F}(P_\lambda)$ is the set of order filters of $P_\lambda$.
Furthermore, using a unique map $\kappa:P_\lambda \rightarrow I$ called coloring, we construct an $I$-colored d-complete poset $(P_\lambda,\le,\kappa,I)$.
Then, there exist a unique order isomorphism $f:W\lambda\rightarrow\mathcal{F}(P_\lambda)$ satisfying the condition that $\mu \rightarrow s_i\mu$ is a cover relation in $W\lambda$ if and only if $f(s_i\mu)\setminus f(\mu)$ consists of one element $x$ with $\kappa(x)=i$ (\cite[Proposition 9.1]{Proc2}).
There are some important applications of these result.
For example, the problem counting the $\lambda$-minuscule elements is reduced to the combinatorial problem counting the ``standard tableaux'' of the corresponding d-complete posets (\cite[Theorem 3.5]{Stem}).
Also, constructing the ``colored hook formula'' for the corresponding from the reflection on Cartan subalgebra's dual space $\Lieh^*$ to the action ``removing hook'' on d-complete poset (\cite{Naka}).

In this paper, we study the relation between the Weyl group orbit through a dominant integral weight and the set of order filters in a d-complete poset in the case that $\Lieg$ is multiply-laced.
To do this, we use the ``folding'' technique (see Section \ref{Sec:folding}).
Assume that $\Lieg$ is of type $A_n,D_n,E_6$.
Let $\sigma$ be a non-trivial automorphism of the Dynkin diagram of $\Lieg$; note that $\sigma$ canonically induces a Lie algebra automorphism of $\Lieg$ and a linear automorphism of $\Lieh^*$.
Then the fixed point subalgebra $\Lieg(0):=\{x\in\Lieg\mid \sigma(x)=x \}$ is isomorphic to a multiply-laced finite-dimensional simple Lie algebra with $\Lieh(0):=\{h\in\Lieh\mid \sigma(h)=h \}$ the Cartan subalgebra.
Let $J$ be the set of $\sigma$-orbits in $I$, and let $\tilde{W}=\langle\tilde{s}_p\mid p\in J\rangle\subset {\it GL}(\Lieh(0)^*)$ be the Weyl group of $\Lieg(0)$.
Then, $\hat{W}:=\{w\in W\mid \sigma w \sigma^{-1}=w\}$ is group isomorphic to $\tilde{W}$.
Let $\text{res}: \Lieh^* \rightarrow \Lieh(0)^* $ be the restriction map.
The map $\text{res}|_{\hat{W}\lambda}$ gives a bijection $\hat{W}\lambda$ onto $\tilde{W}\text{res}(\lambda)$ for a dominant integral weight $\lambda$ of $\Lieg$.
Now, let $\lambda$ be a minuscule dominant integral weight of $\Lieg$.
We define $\tilde{f}:\tilde{W}\text{res}(\lambda)\rightarrow \mathcal{F}(P_\lambda)$ by $\tilde{f}\circ\text{res}=f$, and set $\tilde{\mathcal{F}}(P_{\lambda}):= \text{Im}(\tilde{f})\subset \mathcal{F}(P_\lambda)$.
\begin{theo}[$=$ Theorem \ref{iso fold filter,orbit}; main theorem]
\
\begin{enumerate}\renewcommand{\labelenumi}{(\arabic{enumi})}\label{intro1}
\item[(1)]\ The poset $(\tilde{W}\text{res}(\lambda),\le_w)$ is isomorphic to the poset $(\tilde{\mathcal{F}}(P_{\lambda}),\tilde{\unlhd})$, where $\tilde{\unlhd}$ is a partial order on $\mathcal{F}(P_\lambda)$ defined in terms of an involution $\tilde{S}_p (p\in J)$ on $\mathcal{F}(P_\lambda)$. 
\item[(2)]\ The poset $(\tilde{W}\text{res}(\lambda),\le_s)$ is isomorphic to the poset $(\tilde{\mathcal{F}}(P_{\lambda}),\subseteq)$.
\end{enumerate}
\end{theo}
\noindent In addition, in the case that $\Lieg$ is of type $A_n$, we give an explicit description of $\tilde{\mathcal{F}}(P_{\lambda})$ (see Theorem \ref{Fexplicid}).

This paper is organized as follows. 
In Sections 2 and 3, we explain a (colored) d-complete poset and introduce an involution $S_c$ on $\mathcal{F}(P)$ for a d-complete poset $p$ and a color $c$. 
In Section 4, we fix our notation for finite-dimensional simple Lie algebras and explain the orders $\le_s, \le_w$ on $W\lambda$. 
In Section 5, we explain the fundamental relation between d-complete posets and finite-dimensional simple Lie algebras.
In Section 6, we review the ``folding'' technique for  a simply-laced finite-dimensional simple Lie algebra.
In Section 7, we introduce ``$J$-colored'' d-complete posets by using the folding technique.
In Section 8, we prove Theorem \ref{intro1} above. 
In Section 9, we give an explicit description of $\tilde{\mathcal{F}}(P_{\lambda})$ in the case that $\Lieg$ is of type $A_n$.

\paragraph{Acknowledgements.}
The author would like to thank Professor Daisuke Sagaki, who is his supervisor, for his helpful advice.


\section{d-complete posets.}\label{Sec:d-comp}

Let $(P,\le)$ be a poset.
When $x$ is covered by $y$ in $P$, we write $x\rightarrow y$.
For $x,y\in P$, we set $[x,y]:=\{z\in P\mid x \le z \le y\}$, which we call an \emph{interval}.
A subset $F$ is called an \emph{order filter} if every element in $P$ greater than an element in $F$ is always contained in $F$.
Let $\mathcal{F}(P)$ be the set of all order filters in $P$.
Let $(P,\le)^*$ denote the order dual set of $(P,\le)$.
If $(P,\le)$ is isomorphic, as a poset, to $(P,\le)^*$, then $(P,\le)$ is said to be \emph{self-dual}.
If the Hasse diagram of $P$ is connected, then the $P$ is said to be \emph{connected}.

\begin{defi}[{\cite[Section 2]{Proc1}}]\label{double-tailed}
For $k\ge 3$, we define a poset $d_k(1)$ by the following conditions (1) and (2) (see also Figure \ref{fig:double-tailed}):
\begin{enumerate}\renewcommand{\labelenumi}{(\arabic{enumi})}
\item $d_k(1)$ consists of $2k-2$ elements $w_k,w_{k-1},\cdots,w_3,x,y,z_3,\cdots,z_{k-1},z_k$.
\item The partial order on $d_k(1)$ is as follows:
$$w_k< w_{k-1}< \cdots< w_3,\ w_3< x< z_3,\ w_3< y< z_3,$$
$$\ x\not\le y,\ x\not\ge y,\ z_3< \cdots< z_{k-1}< z_k.$$
\end{enumerate}
We call $d_k(1)$ the \emph{double-tailed diamond}.
Also, we define $d_k^-(1) := d_k(1)\setminus \{z_k\}$ for $k\ge 3$.
\end{defi}

\begin{figure}[h]
\centering
{\unitlength 0.1in%
\begin{picture}(41.8800,12.1500)(0.3600,-13.4300)%
%
\special{pn 8}%
\special{pa 256 512}%
\special{pa 128 640}%
\special{fp}%
\special{pa 128 640}%
\special{pa 256 768}%
\special{fp}%
\special{pa 256 768}%
\special{pa 384 640}%
\special{fp}%
\special{pa 384 640}%
\special{pa 256 512}%
\special{fp}%
\special{pa 768 384}%
\special{pa 768 512}%
\special{fp}%
\special{pa 768 512}%
\special{pa 640 640}%
\special{fp}%
\special{pa 640 640}%
\special{pa 768 768}%
\special{fp}%
\special{pa 768 768}%
\special{pa 768 896}%
\special{fp}%
\special{pa 768 768}%
\special{pa 896 640}%
\special{fp}%
\special{pa 896 640}%
\special{pa 768 512}%
\special{fp}%
\special{pa 1280 256}%
\special{pa 1280 512}%
\special{fp}%
\special{pa 1280 512}%
\special{pa 1152 640}%
\special{fp}%
\special{pa 1152 640}%
\special{pa 1280 768}%
\special{fp}%
\special{pa 1280 768}%
\special{pa 1280 1024}%
\special{fp}%
\special{pa 1280 768}%
\special{pa 1408 640}%
\special{fp}%
\special{pa 1408 640}%
\special{pa 1280 512}%
\special{fp}%
\special{pa 1792 128}%
\special{pa 1792 512}%
\special{fp}%
\special{pa 1792 512}%
\special{pa 1664 640}%
\special{fp}%
\special{pa 1664 640}%
\special{pa 1792 768}%
\special{fp}%
\special{pa 1792 768}%
\special{pa 1792 1152}%
\special{fp}%
\special{pa 1792 768}%
\special{pa 1920 640}%
\special{fp}%
\special{pa 1920 640}%
\special{pa 1792 512}%
\special{fp}%
\special{pa 2432 640}%
\special{pa 2560 768}%
\special{fp}%
\special{pa 2560 768}%
\special{pa 2688 640}%
\special{fp}%
\special{pa 2944 640}%
\special{pa 3072 512}%
\special{fp}%
\special{pa 3072 512}%
\special{pa 3200 640}%
\special{fp}%
\special{pa 3200 640}%
\special{pa 3072 768}%
\special{fp}%
\special{pa 3072 768}%
\special{pa 2944 640}%
\special{fp}%
\special{pa 3072 768}%
\special{pa 3072 896}%
\special{fp}%
\special{pa 3456 640}%
\special{pa 3584 512}%
\special{fp}%
\special{pa 3584 512}%
\special{pa 3584 384}%
\special{fp}%
\special{pa 3584 512}%
\special{pa 3712 640}%
\special{fp}%
\special{pa 3712 640}%
\special{pa 3584 768}%
\special{fp}%
\special{pa 3584 768}%
\special{pa 3456 640}%
\special{fp}%
\special{pa 3584 768}%
\special{pa 3584 1024}%
\special{fp}%
\special{pa 3968 640}%
\special{pa 4096 512}%
\special{fp}%
\special{pa 4096 512}%
\special{pa 4096 256}%
\special{fp}%
\special{pa 4224 640}%
\special{pa 4096 512}%
\special{fp}%
\special{pa 4224 640}%
\special{pa 4096 768}%
\special{fp}%
\special{pa 4096 768}%
\special{pa 3968 640}%
\special{fp}%
\special{pa 4096 768}%
\special{pa 4096 1152}%
\special{fp}%
\put(2.5600,-14.0800){\makebox(0,0){$d_3(1)$}}%
\put(7.6800,-14.0800){\makebox(0,0){$d_4(1)$}}%
\put(12.8000,-14.0800){\makebox(0,0){$d_5(1)$}}%
\put(17.9200,-14.0800){\makebox(0,0){$d_6(1)$}}%
\put(25.6000,-14.0800){\makebox(0,0){$d_3^-(1)$}}%
\put(30.7200,-14.0800){\makebox(0,0){$d_4^-(1)$}}%
\put(35.8400,-14.0800){\makebox(0,0){$d_5^-(1)$}}%
\put(40.9600,-14.0800){\makebox(0,0){$d_6^-(1)$}}%
%
\special{pn 4}%
\special{sh 1}%
\special{ar 256 512 16 16 0 6.2831853}%
\special{sh 1}%
\special{ar 128 640 16 16 0 6.2831853}%
\special{sh 1}%
\special{ar 256 768 16 16 0 6.2831853}%
\special{sh 1}%
\special{ar 384 640 16 16 0 6.2831853}%
\special{sh 1}%
\special{ar 640 640 16 16 0 6.2831853}%
\special{sh 1}%
\special{ar 768 512 16 16 0 6.2831853}%
\special{sh 1}%
\special{ar 768 384 16 16 0 6.2831853}%
\special{sh 1}%
\special{ar 896 640 16 16 0 6.2831853}%
\special{sh 1}%
\special{ar 768 768 16 16 0 6.2831853}%
\special{sh 1}%
\special{ar 768 896 16 16 0 6.2831853}%
\special{sh 1}%
\special{ar 1152 640 16 16 0 6.2831853}%
\special{sh 1}%
\special{ar 1280 512 16 16 0 6.2831853}%
\special{sh 1}%
\special{ar 1280 384 16 16 0 6.2831853}%
\special{sh 1}%
\special{ar 1280 256 16 16 0 6.2831853}%
\special{sh 1}%
\special{ar 1408 640 16 16 0 6.2831853}%
\special{sh 1}%
\special{ar 1280 768 16 16 0 6.2831853}%
\special{sh 1}%
\special{ar 1280 896 16 16 0 6.2831853}%
\special{sh 1}%
\special{ar 1280 1024 16 16 0 6.2831853}%
\special{sh 1}%
\special{ar 1664 640 16 16 0 6.2831853}%
\special{sh 1}%
\special{ar 1792 512 16 16 0 6.2831853}%
\special{sh 1}%
\special{ar 1792 384 16 16 0 6.2831853}%
\special{sh 1}%
\special{ar 1792 256 16 16 0 6.2831853}%
\special{sh 1}%
\special{ar 1792 128 16 16 0 6.2831853}%
\special{sh 1}%
\special{ar 1920 640 16 16 0 6.2831853}%
\special{sh 1}%
\special{ar 1792 768 16 16 0 6.2831853}%
\special{sh 1}%
\special{ar 1792 896 16 16 0 6.2831853}%
\special{sh 1}%
\special{ar 1792 1024 16 16 0 6.2831853}%
\special{sh 1}%
\special{ar 1792 1152 16 16 0 6.2831853}%
\special{sh 1}%
\special{ar 2432 640 16 16 0 6.2831853}%
\special{sh 1}%
\special{ar 2560 768 16 16 0 6.2831853}%
\special{sh 1}%
\special{ar 2688 640 16 16 0 6.2831853}%
\special{sh 1}%
\special{ar 2944 640 16 16 0 6.2831853}%
\special{sh 1}%
\special{ar 3072 512 16 16 0 6.2831853}%
\special{sh 1}%
\special{ar 3200 640 16 16 0 6.2831853}%
\special{sh 1}%
\special{ar 3072 768 16 16 0 6.2831853}%
\special{sh 1}%
\special{ar 3072 896 16 16 0 6.2831853}%
\special{sh 1}%
\special{ar 3456 640 16 16 0 6.2831853}%
\special{sh 1}%
\special{ar 3584 512 16 16 0 6.2831853}%
\special{sh 1}%
\special{ar 3584 384 16 16 0 6.2831853}%
\special{sh 1}%
\special{ar 3712 640 16 16 0 6.2831853}%
\special{sh 1}%
\special{ar 3584 896 16 16 0 6.2831853}%
\special{sh 1}%
\special{ar 3584 768 16 16 0 6.2831853}%
\special{sh 1}%
\special{ar 3584 1024 16 16 0 6.2831853}%
\special{sh 1}%
\special{ar 3968 640 16 16 0 6.2831853}%
\special{sh 1}%
\special{ar 4096 512 16 16 0 6.2831853}%
\special{sh 1}%
\special{ar 4096 384 16 16 0 6.2831853}%
\special{sh 1}%
\special{ar 4096 256 16 16 0 6.2831853}%
\special{sh 1}%
\special{ar 4224 640 16 16 0 6.2831853}%
\special{sh 1}%
\special{ar 4096 768 16 16 0 6.2831853}%
\special{sh 1}%
\special{ar 4096 896 16 16 0 6.2831853}%
\special{sh 1}%
\special{ar 4096 1024 16 16 0 6.2831853}%
\special{sh 1}%
\special{ar 4096 1152 16 16 0 6.2831853}%
\special{sh 1}%
\special{ar 4096 1152 16 16 0 6.2831853}%
\end{picture}}%
\vspace{10pt}
\caption{Double-tailed diamonds.}\label{fig:double-tailed}
\end{figure}
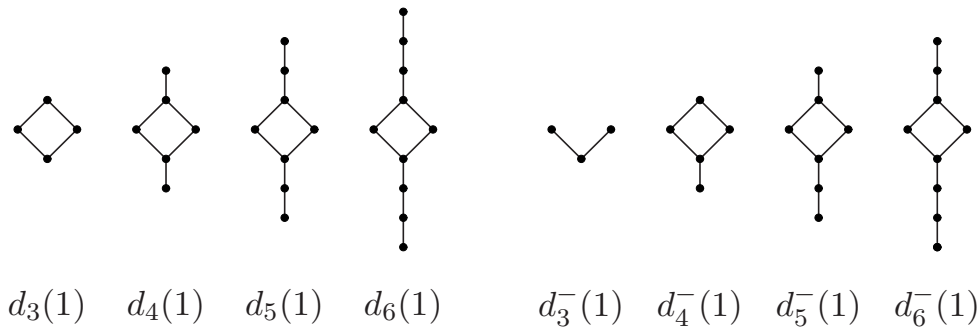

\begin{defi}[{\cite[Section 2]{Proc1}}]\label{dkinterval}
Let $P$ be a poset, and $x,y\in P$.
For $k\ge 3$ (resp., $k\ge4$), if the interval $[x,y]$ is isomorphic to $d_k(1)$ (resp., $d_k^-(1)$), then we say that $[x,y]$ is a \emph{$d_k$-interval} (resp., \emph{$d_k^-$-interval}).
If $w,x,y\in P$ satisfy $w\rightarrow x$ and $w\rightarrow y$, then we say that $\{w,x,y\}$ is a \emph{$d_3^-$-interval}.
\end{defi}

\begin{defi}[{\cite[Section 3]{Proc1}}]\label{special dkinterval}
Let $P$ be a poset.
Let $k\ge 4$ (resp., $k=3$), and let $I=[x,y]$ (resp., $I=\{w,x,y\}$) be a $d_k^-$-interval in $P$.
If $I\cup \{z\}$ is not a $d_k$-interval for any $z\in P$, then the $d_k^-$-interval $I$ is called an \emph{incomplete $d_k^-$-interval}.
If there is another $d_k^-$-interval $I'=[x',y']$ (resp., $I'=\{w',x',y'\}$) such that $I\setminus \{\min I\} = I'\setminus \{\min I'\}$ and $\min I \not= \min I'$, then the $d_k^-$-interval $I$ is called an \emph{overlapping $d_k^-$-interval}.
\end{defi}

\begin{defi}[{\cite[Section 3]{Proc1}}]\label{d-complete}
A finite poset $P$ is called a \emph{d-complete poset} if $P$ satisfies the following conditions (D1)-(D3):
\begin{enumerate}\renewcommand{\labelenumi}{(D\arabic{enumi})}
\item There is no incomplete $d_k^-$-interval in $P$ for any $k\ge 3$.
\item If $I$ is a $d_k$-interval in $P$ for some $k\ge 3$, then there is no element that is not included in $I$ and is covered by $\max I$.
\item There is no overlapping $d_k^-$-interval in $P$ for any $k\ge 3$.
\end{enumerate}
\end{defi}

\begin{figure}[h]
\centering
\input{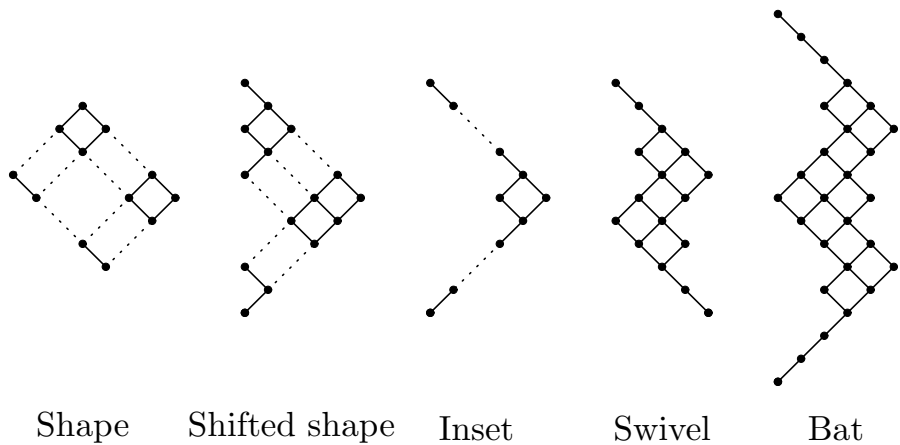}
\vspace{10pt}
\caption{Connected, self-dual d-complete posets.}\label{fig:self-dual}
\end{figure}

\begin{defi}[{\cite[Section 4]{Proc1}}]\label{top tree}
Let $P$ be a d-complete poset.
We define the \emph{top tree} $T_P$ of $P$ to be the subset of $P$ consisting of all elements $x\in P$ satisfying the condition that
\begin{enumerate}\renewcommand{\labelenumi}{(T)}
\item $\# \{ z\in P\mid y\rightarrow z\}\le 1$ for every $y\in P$ such that $x\le y$.
\end{enumerate}
\end{defi}

\begin{prop}[{\cite[Sections 3 and 14]{Proc1}},{\cite[Proposition 8.6]{Proc2}}]\label{property d-comp}
Let $P$ be a d-complete poset.
\begin{enumerate}\renewcommand{\labelenumi}{(\arabic{enumi})}
\item If $P$ is connected, then $P$ has a unique maximum element.
\item For each $w\in P\setminus T_P$,\ there are unique $z\in P$ and $k\ge 3$ such that $[w,z]$ is a $d_k$-interval.
\item A connected self-dual d-complete poset is isomorphic, as a poset, to one of those in Figure \ref{fig:self-dual}.
\end{enumerate}
\end{prop}

\begin{exam}\label{diagram}
(1)\ For $m,n\ge 1$, we set $Y_{m,n}:=\{(i,j)\mid i,j\in\mathbb{Z},\ 1\le i\le m,\ 1\le j\le n\}$; we identify $Y_{m,n}$ with a Young diagram of rectangular shape as the left diagram in Figure \ref{fig:diagram1}.
We define a partial order on $Y_{m,n}$ as follows.
If $i_1\ge i_2$ and $j_1\ge j_2$, then $(i_1,j_1)\le(i_2,j_2)$.
Then the poset $(Y_{m,n},\le)$ is a d-complete poset of Shape class in Figure \ref{fig:self-dual}.
The top tree $T_{Y_{m,n}}$ of $Y_{m,n}$ is identical to the set of those cells in the first row or in the first column; see the right diagram in Figure \ref{fig:diagram1}.

\ytableausetup{mathmode,boxsize=2.4em}
\begin{figure}[h]
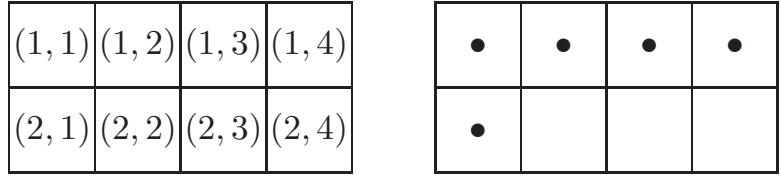

\begin{center}
\begin{ytableau}
(1,1)&(1,2)&(1,3)&(1,4)\\
(2,1)&(2,2)&(2,3)&(2,4)
\end{ytableau}
\qquad
\begin{ytableau}
\bullet&\bullet&\bullet&\bullet\\
\bullet&&&
\end{ytableau}
\end{center}
\caption{The Young diagram corresponding to $Y_{2,4}$ and its top tree.}\label{fig:diagram1}
\end{figure}

(2)\ For $n\ge 1$, we set $SY_{n}=\{(i,j)\mid i,j\in\mathbb{Z},\ 1\le i\le n,i\le j\le n\}$; we identify $SY_{n}$ with a shifted Young diagram of ``triangular shape'' as the left diagram in Figure \ref{fig:diagram2}.
We define a partial order on $SY_{n}$ as that on $Y_{m,n}$.
Then the poset $(SY_{n},\le)$ is a d-complete poset of Shifted Shape class in Figure \ref{fig:self-dual}.
The top tree $T_{SY_{n}}$ of $SY_{n}$ is identical to the set of those cells in the first row or in the second column; see the right diagram in Figure \ref{fig:diagram2}.

\ytableausetup{mathmode,boxsize=2.4em}
\begin{figure}[h]
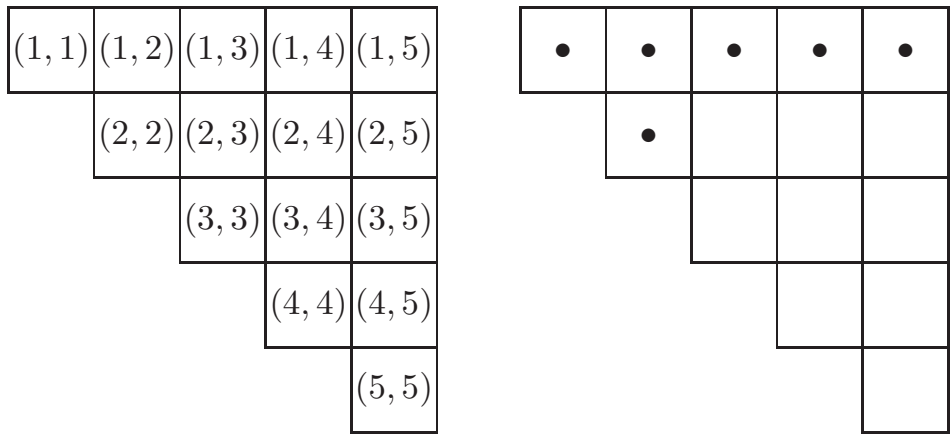

\begin{center}
\begin{ytableau}
(1,1)&(1,2)&(1,3)&(1,4)&(1,5)\\
\none&(2,2)&(2,3)&(2,4)&(2,5)\\
\none&\none&(3,3)&(3,4)&(3,5)\\
\none&\none&\none&(4,4)&(4,5)\\
\none&\none&\none&\none&(5,5)
\end{ytableau}
\qquad
\begin{ytableau}
\bullet&\bullet&\bullet&\bullet&\bullet\\
\none&\bullet&&&\\
\none&\none&&&\\
\none&\none&\none&&\\
\none&\none&\none&\none&
\end{ytableau}
\end{center}
\caption{The shifted Young diagram corresponding to $SY_{5}$ and its top tree.}\label{fig:diagram2}
\end{figure}
\end{exam}

In what follows, we use Young diagrams and shifted Young diagrams for d-complete posets of Shape and Shifted Shape classes.
For a given subset $X$ in these d-complete posets $P$, we indicate an element in $X$ (resp., in $P\setminus X$) by a white cell (resp., gray cell).
For example, the left diagram in Figure \ref{fig:filter diagram} indicates the subset $\{(1,1),(1,2),(1,3),(2,1)\}$ of $Y_{2,4}$, which is in fact an order filter of $Y_{2,4}$.
The right diagram in Figure \ref{fig:filter diagram} indicates the subset $\{(1,1),(1,2),(1,3),(1,4),(2,2),(2,3),(3,3)\}$ of $SY_{5}$, which is in fact an order filter of $SY_{5}$.

\ytableausetup{mathmode,boxsize=1.6em}
\begin{figure}[h]
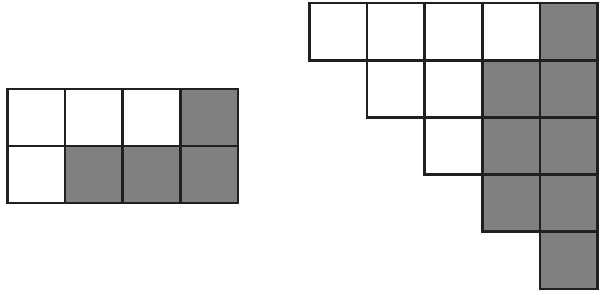

\begin{center}
\ydiagram[*(gray)]{3+1,1+3}*[*(white)]{4,4}\qquad \ydiagram[*(gray)]{4+1,3+2,3+2,3+2,4+1}*[*(white)]{5,1+4,2+3,3+2,4+1}
\end{center}
\caption{Examples of order filters of  d-complete posets.}\label{fig:filter diagram}
\end{figure}


\section{Colored d-complete posets and involutions on $\mathcal{F}(P)$.}\label{Sec:color}

Let $(P,\le)$ be a poset, and let $C$ be a set.
We call a map $\kappa:P\rightarrow C$ a \emph{coloring} of $P$ with $C$ the set of \emph{colors}
 the quadruple $(P,\le,\kappa,C)$ a \emph{colored poset}.

\begin{prop}[{\cite[Proposition 8.6]{Proc2}}]\label{colored d-comp}
Let $(P,\le)$ be a d-complete poset, and let $C$ be a set such that $\#C=\#T_P$.
There exists a coloring $\kappa:P\rightarrow C$ of $P$ satisfying the following conditions (a) and (b):
\begin{enumerate}\renewcommand{\labelenumi}{(\alph{enumi})}
\item The restriction of $\kappa:P\rightarrow C$ to the top tree $T_P$ is a bijection from  $T_P$ onto $C$.
Namely, each element of $T_P$ has a different color from each other.
\item If $[w,z]$ is a $d_k$-interval for some $k\ge 3$, then $\kappa(w)=\kappa(z)$.
\end{enumerate}
Moreover, this coloring of $P$ with $C$ the set of colors is unique, up to the coloring of the top tree $T_P$ in (a).
In this case, we call the quadruple $(P,\le,\kappa,C)$ a \emph{colored d-complete poset}.
\end{prop}

\ytableausetup{mathmode,boxsize=1.6em}
\begin{figure}[h]
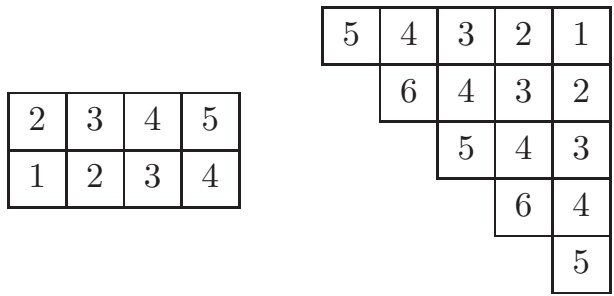

\begin{center}
\begin{ytableau}
2&3&4&5\\
1&2&3&4
\end{ytableau}
\qquad 
\begin{ytableau}
5&4&3&2&1\\
\none&6&4&3&2\\
\none&\none&5&4&3\\
\none&\none&\none&6&4\\
\none&\none&\none&\none&5
\end{ytableau}
\end{center}
\caption{Colored d-complete posets.}\label{fig:c-d-comp}
\end{figure}

\begin{prop}[{\cite[Section 3]{Proc2}}]\label{property c-d-comp}
Let $(P,\le,\kappa,C)$ be a colored d-complete poset.
\begin{enumerate}\renewcommand{\labelenumi}{(\arabic{enumi})}
\item[(1)]\ Let $x,y\in P$.
If there is the covering relation between $x$ and $y$, or if $x$ and $y$ are incomparable, then $\kappa(x)\not=\kappa(y)$, that is, $x$ and $y$ have distinct colors.
\item[(2)]\ Let $I$ be an interval of $P$.
If $I$ is a totally order set, then $\kappa(x)\not=\kappa(y)$ for all elements $x,y\in I$ with $x\not=y$, that is, each element in $I$ has a distinct color from each other.
\item[(3)]\ For each $c\in C$, the subset $\kappa^{-1}(\{c\})$ consisting of elements in $P$ having the color $c$ is a totally order set.
\end{enumerate}
\end{prop}

\begin{defi}\label{filter involution}
Let $(P,\le,\kappa,C)$ be a finite colored poset.
For each $c\in C$, we define maps $A_c,R_c,S_c: \mathcal{F}(P)\rightarrow \mathcal{F}(P)$ as follows.
For each $F\in \mathcal{F}(P)$,
$$A_c(F):=\bigcup_{\substack{F'\in \mathcal{F}(P) \\ F'\setminus F\subseteq \kappa^{-1}(\{c\})}}F' ,\qquad\qquad R_c(F):=\bigcap_{\substack{F'\in \mathcal{F}(P) \\ F\setminus F'\subseteq \kappa^{-1}(\{c\})}}F' ,$$
$$S_c(F):=
\begin{cases}
    (A_c(F)\setminus F)\cup R_c(F) & \text{if }(A_c(F)\setminus F)\cup R_c(F) \in\mathcal{F}(P), \\
    F & \textrm{otherwise}.
\end{cases}$$
\end{defi}

\begin{rema}\label{remark f-invo}
It is obvious by the definition that $A_c(F)\supseteq F \supseteq R_c(F)$.
If $F$ satisfies $R_c(F)=F$ (resp., $A_c(F)=F$), then $S_c(F)=A_c(F)$ (resp., $S_c(F)=R_c(F)$).
Also, it can be easily verified that $A_c(F)\supseteq S_c(F) \supseteq R_c(F)$.
\end{rema}

\begin{exam}\label{exam f-invo}
\ytableausetup{mathmode,boxsize=1.6em}
\begin{spacing}{2.5}
Let $P=Y_{2,4}=\ydiagram{4,4}$\ , and define a coloring $\kappa:P\rightarrow\{1,2,3\}$ for $P$ by
$\begin{ytableau}
2&3&2&1\\
1&2&3&2
\end{ytableau}$\ .
Let $F=$
$\begin{ytableau}
2&3&2&*(gray)1\\
1&*(gray)2&*(gray)3&*(gray)2
\end{ytableau}$\ ;
notice that $F$ is an order filter of $P$.
Then,\ $A_2(F),R_2(F),S_2(F)$ are as follows:
\end{spacing}

$$A_2\left(\vbox to 25pt{} 
\begin{ytableau}
2&3&2&*(gray)1\\
1&*(gray)2&*(gray)3&*(gray)2
\end{ytableau}
\right)=
\begin{ytableau}
2&3&2&*(gray)1\\
1&2&*(gray)3&*(gray)2
\end{ytableau}\ ,$$

$$R_2\left(\vbox to 25pt{} 
\begin{ytableau}
2&3&2&*(gray)1\\
1&*(gray)2&*(gray)3&*(gray)2
\end{ytableau}
\right)=
\begin{ytableau}
2&3&*(gray)2&*(gray)1\\
1&*(gray)2&*(gray)3&*(gray)2
\end{ytableau}\ ,$$

$$S_2\left(\vbox to 25pt{} 
\begin{ytableau}
2&3&2&*(gray)1\\
1&*(gray)2&*(gray)3&*(gray)2
\end{ytableau}
\right)=
\begin{ytableau}
2&3&*(gray)2&*(gray)1\\
1&2&*(gray)3&*(gray)2
\end{ytableau}\ .$$

\end{exam}

\begin{lemm}\label{property f-invo}
Let $(P,\le,\kappa,C)$ be a colored poset.
For every $F\in\mathcal{F}(P)$ and $c\in C$, the following hold.
\begin{enumerate}\renewcommand{\labelenumi}{(\arabic{enumi})}
\item $A_c(S_c(F))=A_c(F)$.
\item $R_c(S_c(F))=R_c(F)$.
\item $S_c(S_c(F))=F$. Namely, the map $S_c:\mathcal{F}(P)\rightarrow\mathcal{F}(P)$ is an involution on $\mathcal{F}(P)$.
\end{enumerate}
\end{lemm}

\begin{proof}
By the definition of $S_c(F)$, it suffices to consider the case that $(A_c(F)\setminus F)\cup R_c(F)$ is an order filter of $P$.
\begin{enumerate}\renewcommand{\labelenumi}{(\arabic{enumi})}
\item Since all elements of $A_c(F)\setminus R_c(F)$ have the color $c$ and since $S_c(F)\supseteq R_c(F)$, all elements of $A_c(F)\setminus S_c(F)$ also have the color $c$.
Hence,\ $A_c(F)\in\{F'\in \mathcal{F}(P)\mid F'\setminus S_c(F)\subseteq \kappa^{-1}(\{c\})\}$, and hence $A_c(S_c(F))\supseteq A_c(F)$ by the definition of $A_c$.
This inclusion relation also implies that all elements in $A_c(S_c(F))\setminus A_c(F)$ have the color $c$.
By the definition of $A_c$, all elements in $A_c(F)\setminus F$ have the color $c$.
Hence, $A_c(S_c(F))\in\{F'\in \mathcal{F}(P)\mid F'\setminus F\subseteq \kappa^{-1}(\{c\})\}$.
By the definition of $A_c$, we obtain $A_c(S_c(F))\subseteq A_c(F)$.
Therefore,\ $A_c(S_c(F))= A_c(F)$.
\item Similar to Part (1).
\item We compute
\begin{eqnarray*}
&&(A_c(S_c(F))\setminus S_c(F))\cup R_c(S_c(F))\\
&=&(A_c(F)\setminus ((A_c(F)\setminus F)\cup R_c(F)))\cup R_c(F)\\
&=&((A_c(F)\setminus (A_c(F)\setminus F))\cap (A_c(F)\setminus R_c(F)))\cup R_c(F)\\
&=&(F\cup R_c(F))\cap ((A_c(F)\setminus R_c(F))\cup R_c(F))\\
&=&F\cap A_c(F)\\
&=&F
\end{eqnarray*}
Therefore,\ $(A_c(S_c(F))\setminus S_c(F))\cup R_c(S_c(F))$ is an order filter of $P$, and $S_c(S_c(F))=F$.
\end{enumerate}
\end{proof}

\begin{defi}[]\label{S-order}
Let $(P,\le,\kappa,C)$ be a colored poset.
We define an order $\unlhd$ on $\mathcal{F}(P)$ as follows.
For $F,F'\in \mathcal{F}(P)$, $F\unlhd F'$ if there exists a sequence of order filters $F=F_0,F_1,\ldots,F_{n-1},F_n=F'$ such that for all $i\in \{0,1,\ldots ,n-1\}$, there exist $c_i\in C$ such that $S_{c_i}(F_i)=F_{i+1}\supset F_i$.
\end{defi}

\begin{lemm}\label{property f-invo,c-d-comp}
Let $(P,\le,\kappa,C)$ be a colored d-complete poset.
For an order filter $F$ of $P$ and a color $c\in C$, the symmetric difference of $F$ and $S_c(F)$ has at most one element.
\end{lemm}

\begin{proof}
Suppose, for a contradiction, that the symmetric difference of $F$ and $S_c(F)$ has more than one element.
Let $x,y$ be the elements of the symmetric difference, with $x\ne y$.
Because both $x$ and $y$ have the color $c$, it follows from Proposition \ref{property c-d-comp}(3) that either $x<y$ or $x>y$ holds;
we may assume that $x<y$.
Because both $F$ and $S_c(F)$ are order filters, we deduce that either $x,y\in F\setminus S_c(F)$ or $x,y\in S_c(F)\setminus F$ holds.
Assume that $x,y \in F\setminus S_c(F)$.
Since $x<y$, there exists an element $z\in P$ such that $x\rightarrow z$ and $z\le y$.
Because $x\in F$, and $F$ is an order filter, we see that $z\in F$.
Similarly, because $y\notin S_c(F)$, and $S_c(F)$ is an order filter, we see that $z\notin S_c(F)$.
Thus we get $z\in F\setminus S_c(F)$; in particular, $z$ has the color $c$.
However, this contradicts Proposition \ref{property c-d-comp}(1); recall that $x\rightarrow z$, and $x$ has the color $c$.
A proof for the case that $x,y\in S_c(F)\setminus F$ is similar.
 \end{proof}

\begin{rema}\label{S-order,inclusion}
Let $(P,\le,\kappa,C)$ be a colored d-complete poset.
By Lemma \ref{property f-invo,c-d-comp}, it is clear that for $F,F'\in \mathcal{F}(P)$, $F\subseteq F'$ if and only if $F\unlhd F'$.
In particular, $(\mathcal{F}(P),\subseteq)$ and $(\mathcal{F}(P),\unlhd)$ are order isomorphic.
\end{rema}


\section{Finite-dimensional simple Lie algebras.}\label{Sec:Lie}

Let $\Lieg=\Lieg(A)$ be a finite-dimensional simple Lie algebra over $\mathbb{C}$, with $A=(a_{ij})_{i,j\in I}$ the Cartan matrix.
Denote by $\Lieh$ a Cartan subalgebra of $\Lieg$,\ $\Pi^\vee=\{h_i\mid i\in I\}\subset\Lieh$ the set of simple coroots, \ $\Pi=\{\alpha_i\mid i\in I\}\subset\Lieh^*$ the set of simple roots, \ $\Delta_+\subset \Lieh^*$ the set of positive roots, \ $\Delta_-\subset \Lieh^*$ the set of negative roots,\ $\Lambda_i\in\Lieh^*(i\in I)$ the fundamental weights, and $e_i,f_i\in\Lieg (i\in I)$ the Chevalley generators.
Let $W=\langle s_i\mid i\in I\rangle$ be the Weyl group of $\Lieg$, where $s_i$ is the simple reflection in $\alpha_i$ for $i\in I$.
For $\beta\in\Delta_+$, $\beta^\vee\in \Lieh$ denotes the dual root of $\beta$, and $s_\beta\in W$ denotes the reflection in $\beta$; recall that if $\beta  = w(\beta')$ for $\beta' \in \Delta_+$ and $w \in W$, then $s_\beta = s_{w(\beta')} = ws_{\beta'}w^{-1}$.

\begin{defi}[]\label{strong-order}
Let $\lambda$ be a dominant integral weight of $\Lieg$.
We define the order $\le_s$ on the Weyl group orbit $W\lambda$ through $\lambda$ as follows.
For $\mu,\mu'\in W\lambda$, $\mu\le_s\mu'$ if there exists a finite sequence $\mu=\mu_0,\mu_1,\ldots,\mu_{k-1},\mu_k=\mu'$ of elements in $W\lambda$ and a finite sequence $\beta_0,\ldots,\beta_{k-1}$ of elements in $\Delta_+$ such that $s_{\beta_i}(\mu_i)=\mu_{i+1}$ and $\mu_i(\beta_i^\vee)>0$ for each $i\in \{0,1,\ldots k-1\}$.
\end{defi}

\begin{lemm}[]\label{1strong-ex}
Let $\mu$ be an integral weight of $\Lieg$, and $\beta\in\Delta_+$.
For $w \in W$, if $\mu <_s s_\beta(\mu)$ and $w(\beta)\in\Delta_+$, then $w(\mu) <_s w s_\beta(\mu)$.
\end{lemm}
\begin{proof}
Since $s_{w(\beta)}(w \mu) = w s_\beta w^{-1}(w\mu) = w s_\beta(\mu)$, and since $w(\mu) \not= w s_\beta(\mu)$, either $w(\mu) <_s w s_\beta(\mu)$ or $w(\mu) >_s w s_\beta(\mu)$ holds.
By the definition of $\le_s$, there exists $n\in \mathbb{Z}_{>0}$ such that  $s_\beta(\mu) = \mu - n\beta$.
Thus, $w s_\beta(\mu) = w (\mu - n\beta) = w(\mu) - nw(\beta)$.
Because $w(\beta)\in\Delta_+$, we obtain $w(\mu) <_s w s_\beta(\mu)$, as desired.
\end{proof}

\begin{prop}[{\cite[Lemma 4.1]{Litt}}]\label{strong-slide}
Let $\mu_1, \mu_2 \in W\lambda$, and $i \in I$.
\begin{enumerate}\renewcommand{\labelenumi}{(\arabic{enumi})}
\item[(1)] If $\mu_1 \le_s \mu_2$ , $\mu_1(h_i) \ge 0$ and $\mu_2(h_i) \le 0$, then $\mu_1 \le_s s_i(\mu_2)$.
\item[(2)] If $\mu_1 \le_s \mu_2$ , $\mu_1(h_i) \ge 0$ and $\mu_2(h_i) \le 0$, then $s_i(\mu_1) \le_s \mu_2$.
\item[(3)] If $\mu_1 \le_s \mu_2$ , $\mu_1(h_i) \le 0$ and $\mu_2(h_i) \le 0$, then $s_i(\mu_1) \le_s s_i(\mu_2)$.
\item[(4)] If $\mu_1 \le_s \mu_2$ , $\mu_1(h_i) \ge 0$ and $\mu_2(h_i) \ge 0$, then $s_i(\mu_1) \le_s s_i(\mu_2)$.
\end{enumerate}
\end{prop}

\color{black}

\begin{defi}[]\label{weak-order}
Let $\lambda$ be a dominant integral weight of $\Lieg$.
We define the order $\le_w$ on $W\lambda$ as follows.
For $\mu,\mu'\in W\lambda$, $\mu\le_w\mu'$ if there exists a finite sequence $\mu=\mu_0,\mu_1,\ldots,\mu_{k-1},\mu_k=\mu'$ of elements in $W\lambda$ and a finite sequence $j_0,\ldots,j_{k-1}$ of elements in $I$ such that $s_{j_i}(\mu_i)=\mu_{i+1}$ and $\mu_i(h_j)>0$ for each $i\in \{0,1,\ldots ,k-1\}$.
\end{defi}

\begin{rema}[see, e.g., {\cite[Section 4.3]{Gree} and \cite[Section 2.4]{Bjor}}]\label{bruhat}
Let $\lambda$ be a dominant integral weight, and $W_\lambda:=\{w\in W\mid w\lambda=\lambda\}$ the stabilizer of $\lambda$; we have the canonical bijection $W/W_\lambda\rightarrow W\lambda, wW_\lambda\mapsto w\lambda$.
It is known that $W_\lambda$ is the subgroup of $W$ generated by $s_i$ for $i\in I$ such that $\lambda(h_i)=0$, and each coset in $W/W_\lambda$ has a unique element whose length is minimal among the element in the coset; we regard $W/W_\lambda$ as a subset of $W$ by taking the minimal-length coset representative from each coset in $W/W_\lambda$.
The poset $W/W_\lambda$ in the restriction of the Bruhat order (resp., the weak Bruhat order) on $W$ is order isomorphic to $(W\lambda,\le_s)$ (resp., $(W\lambda,\le_w)$) under the canonical map $W/W_\lambda\rightarrow W\lambda$ above. 
\end{rema}


\section{Order isomorphism between $W\lambda$ and $\mathcal{F}(P)$.}\label{Sec:Lie-d-comp}


Let $\Lieg$ be a finite-dimensional simple Lie algebra over $\mathbb{C}$.

\begin{defi}[]\label{minuscule weight}
Let $\lambda$ be a dominant integral weight of $\Lieg$.
We call $\lambda$ a \emph{minuscule weight} if $\lambda$ satisfies $(w\lambda)(h_i)\in\{-1,0,1\}$ for all $w\in W$ and $i\in I$.
\end{defi}

Table \ref{tab:minuscule} below is the list of minuscule weights of simply-laced finite-dimensional simple Lie algebras; the vertices of the Dynkin diagram are numbered as Figure \ref{fig:dynkin index1}.

\vspace{10pt}
\begin{table}[h]
\begin{center}
\begin{tabular}{|c|c|c|}
\hline
$\Lieg$&minuscule weight $\lambda$\\
\hline
$A_n$&$\Lambda_1,\ldots,\Lambda_n$\\
\hline
$D_n$&$\Lambda_1,\Lambda_{n-1},\Lambda_n$\\
\hline
$E_6$&$\Lambda_1,\Lambda_5$\\
\hline
$E_7$&$\Lambda_6$\\
\hline
$E_8$&none\\
\hline
\end{tabular}
\caption{Minuscule weights; simply-laced case.}\label{tab:minuscule}
\end{center}
\end{table}

\begin{figure}[h]
\centering
\vspace{10pt}
{\unitlength 0.1in%
\begin{picture}(45.8500,28.7600)(-2.6500,-29.7100)%
%
\special{pn 4}%
\special{sh 1}%
\special{ar 320 480 16 16 0 6.2831853}%
\special{sh 1}%
\special{ar 640 480 16 16 0 6.2831853}%
\special{sh 1}%
\special{ar 960 480 16 16 0 6.2831853}%
\special{sh 1}%
\special{ar 1600 480 16 16 0 6.2831853}%
\special{sh 1}%
\special{ar 1920 480 16 16 0 6.2831853}%
\special{sh 1}%
\special{ar 1920 480 16 16 0 6.2831853}%
%
\special{pn 8}%
\special{pa 320 480}%
\special{pa 963 480}%
\special{fp}%
\special{pa 1606 480}%
\special{pa 1920 480}%
\special{fp}%
%
\special{pn 8}%
\special{pa 960 480}%
\special{pa 1600 480}%
\special{dt 0.045}%
\put(3.2000,-3.2000){\makebox(0,0){\tiny$1$}}%
\put(6.4000,-3.2000){\makebox(0,0){\tiny$2$}}%
\put(9.6000,-3.2000){\makebox(0,0){\tiny$3$}}%
\put(27.2000,-1.6000){\makebox(0,0){\tiny$1$}}%
\put(4.8000,-12.8000){\makebox(0,0){\tiny$1$}}%
\put(27.2000,-12.8000){\makebox(0,0){\tiny$1$}}%
\put(30.4000,-1.6000){\makebox(0,0){\tiny$2$}}%
\put(8.0000,-12.8000){\makebox(0,0){\tiny$2$}}%
\put(30.4000,-12.8000){\makebox(0,0){\tiny$2$}}%
\put(33.6000,-1.6000){\makebox(0,0){\tiny$3$}}%
\put(11.2000,-12.8000){\makebox(0,0){\tiny$3$}}%
\put(33.6000,-12.8000){\makebox(0,0){\tiny$3$}}%
\put(36.8000,-12.8000){\makebox(0,0){\tiny$4$}}%
\put(14.4000,-12.8000){\makebox(0,0){\tiny$4$}}%
\put(17.6000,-12.8000){\makebox(0,0){\tiny$5$}}%
\put(40.0000,-12.8000){\makebox(0,0){\tiny$5$}}%
\put(43.2000,-12.8000){\makebox(0,0){\tiny$6$}}%
\put(11.2000,-19.2000){\makebox(0,0){\tiny$6$}}%
\put(33.6000,-19.2000){\makebox(0,0){\tiny$7$}}%
\put(16.0000,-3.2000){\makebox(0,0){\tiny$n-1$}}%
\put(43.2000,-1.6000){\makebox(0,0){\tiny$n-1$}}%
\put(40.0000,-1.6000){\makebox(0,0){\tiny$n-2$}}%
\put(19.2000,-3.2000){\makebox(0,0){\tiny$n$}}%
\put(40.0000,-8.0000){\makebox(0,0){\tiny$n$}}%
\put(0.6000,-4.8000){\makebox(0,0){\large$A_n:$}}%
\put(24.6000,-4.8000){\makebox(0,0){\large$D_n:$}}%
\put(2.2000,-16.0000){\makebox(0,0){\large$E_6:$}}%
\put(24.6000,-16.0000){\makebox(0,0){\large$E_7:$}}%
%
\special{pn 4}%
\special{sh 1}%
\special{ar 480 1440 16 16 0 6.2831853}%
\special{sh 1}%
\special{ar 800 1440 16 16 0 6.2831853}%
\special{sh 1}%
\special{ar 1120 1440 16 16 0 6.2831853}%
\special{sh 1}%
\special{ar 1440 1440 16 16 0 6.2831853}%
\special{sh 1}%
\special{ar 1760 1440 16 16 0 6.2831853}%
\special{sh 1}%
\special{ar 1120 1760 16 16 0 6.2831853}%
%
\special{pn 8}%
\special{pa 2720 1440}%
\special{pa 4320 1440}%
\special{fp}%
\special{pa 3360 1440}%
\special{pa 3360 1760}%
\special{fp}%
%
\special{pn 8}%
\special{pa 480 1440}%
\special{pa 1760 1440}%
\special{fp}%
\special{pa 1120 1440}%
\special{pa 1120 1760}%
\special{fp}%
%
\special{pn 4}%
\special{sh 1}%
\special{ar 2720 320 16 16 0 6.2831853}%
\special{sh 1}%
\special{ar 3040 320 16 16 0 6.2831853}%
\special{sh 1}%
\special{ar 3360 320 16 16 0 6.2831853}%
\special{sh 1}%
\special{ar 4000 320 16 16 0 6.2831853}%
\special{sh 1}%
\special{ar 4320 320 16 16 0 6.2831853}%
\special{sh 1}%
\special{ar 4000 640 16 16 0 6.2831853}%
\special{sh 1}%
\special{ar 4000 640 16 16 0 6.2831853}%
%
\special{pn 8}%
\special{pa 2720 320}%
\special{pa 3360 320}%
\special{fp}%
\special{pa 4000 320}%
\special{pa 4320 320}%
\special{fp}%
\special{pa 4000 320}%
\special{pa 4000 640}%
\special{fp}%
%
\special{pn 8}%
\special{pa 3360 320}%
\special{pa 4000 320}%
\special{dt 0.045}%
%
\special{pn 4}%
\special{sh 1}%
\special{ar 1356 2556 16 16 0 6.2831853}%
\special{sh 1}%
\special{ar 1676 2556 16 16 0 6.2831853}%
\special{sh 1}%
\special{ar 1996 2556 16 16 0 6.2831853}%
\special{sh 1}%
\special{ar 2316 2556 16 16 0 6.2831853}%
\special{sh 1}%
\special{ar 2636 2556 16 16 0 6.2831853}%
\special{sh 1}%
\special{ar 2956 2556 16 16 0 6.2831853}%
\special{sh 1}%
\special{ar 3276 2556 16 16 0 6.2831853}%
\special{sh 1}%
\special{ar 2636 2876 16 16 0 6.2831853}%
\special{sh 1}%
\special{ar 2636 2876 16 16 0 6.2831853}%
%
\special{pn 8}%
\special{pa 1356 2556}%
\special{pa 3276 2556}%
\special{fp}%
\special{pa 2636 2556}%
\special{pa 2636 2876}%
\special{fp}%
\put(13.5600,-23.9600){\makebox(0,0){\tiny$1$}}%
\put(16.7600,-23.9600){\makebox(0,0){\tiny$2$}}%
\put(19.9600,-23.9600){\makebox(0,0){\tiny$3$}}%
\put(23.1600,-23.9600){\makebox(0,0){\tiny$4$}}%
\put(26.3600,-23.9600){\makebox(0,0){\tiny$5$}}%
\put(29.5600,-23.9600){\makebox(0,0){\tiny$6$}}%
\put(32.7600,-23.9600){\makebox(0,0){\tiny$7$}}%
\put(26.3600,-30.3600){\makebox(0,0){\tiny$8$}}%
\put(10.9600,-27.1600){\makebox(0,0){\large$E_8:$}}%
%
\special{pn 4}%
\special{sh 1}%
\special{ar 2720 1440 16 16 0 6.2831853}%
\special{sh 1}%
\special{ar 3040 1440 16 16 0 6.2831853}%
\special{sh 1}%
\special{ar 3360 1440 16 16 0 6.2831853}%
\special{sh 1}%
\special{ar 3680 1440 16 16 0 6.2831853}%
\special{sh 1}%
\special{ar 4000 1440 16 16 0 6.2831853}%
\special{sh 1}%
\special{ar 4320 1440 16 16 0 6.2831853}%
\special{sh 1}%
\special{ar 3360 1760 16 16 0 6.2831853}%
\special{sh 1}%
\special{ar 3360 1760 16 16 0 6.2831853}%
\end{picture}}%
\vspace{10pt}
\caption{Simply-laced Dynkin diagrams.}\label{fig:dynkin index1}
\end{figure}
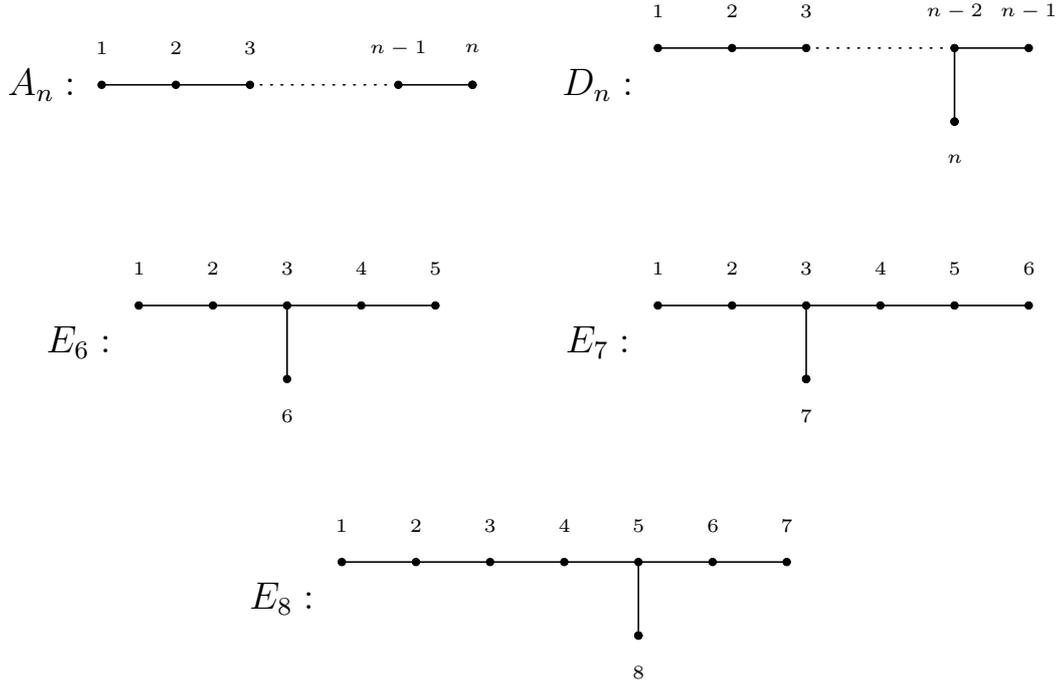

\begin{rema}[{\cite[Lemma 11.1.18]{Gree}} and Remark \ref{bruhat}]\label{strong,weak}
Assume that $\lambda$ is minuscule.
For $\mu,\mu'\in W\lambda$, $\mu\le_s \mu'$ if and only if $\mu\le_w \mu'$.
Therefore, $(W\lambda,\le_s)$ and $(W\lambda,\le_w)$ are order isomorphic.
\end{rema}

\begin{prop}[{\cite[Section 14]{Proc1}}]\label{iso filter,orbit}
Assume that $\Lieg$ is simply-laced.
Let $\lambda$ be a minuscule weight of $\Lieg$.
There exists a connected, self-dual d-complete poset $P_{\lambda}$ such that $(W\lambda,\le_s)$ and $(\mathcal{F}(P_{\lambda}),\subseteq)$ are isomorphic, as posets (see also Table \ref{tab:iso filter,orbit}).
\end{prop}

\begin{table}[h]
\begin{center}
\begin{tabular}{|c|c|c|}
\hline
$\Lieg$&minuscule weight $\lambda$&corresponding d-complete poset $P_{\lambda}$\\
\hline
$A_n$&$\Lambda_i (1\le i\le n)$&$Y_{i,n-i+1}$ (Shape)\\
\hline
$D_n$&$\Lambda_1$&$SY_{n-1}$ (Shifted Shape)\\
\hline
$D_n$&$\Lambda_{n-1},\Lambda_n$&$d_n(1)$ (Inset)\\
\hline
$E_6$&$\Lambda_1,\Lambda_5$&Swivel\\
\hline
$E_7$&$\Lambda_6$&Bat\\
\hline
\end{tabular}
\caption{The d-complete posets $P_{\lambda}$ corresponding to minuscule weights $\lambda$.}\label{tab:iso filter,orbit}
\end{center}
\end{table}

Keep the setting in Proposition \ref{iso filter,orbit}, with $\lambda=\Lambda_i$ for some $i\in I$ such that $\Lambda_i$ is minuscule.
We know from \cite[Proposition 8.6]{Proc2} that the graph obtained from the Hasse diagram of the top tree\ $T_{P_{\lambda}}$ of $P_\lambda$ by replacing each allow by an edge is identical to the Dynkin diagram of $\Lieg$; in particular, $\#I=\#T_{P_{\lambda}}$.
By Proposition \ref{colored d-comp}, we can obtain the colored poset\ $(P_{\lambda},\le,\kappa,I)$ such that $\kappa|_{T_{P_{\lambda}}}:T_{P_{\lambda}}\rightarrow I$ is the graph isomorphism and the maximum element of $P_{\lambda}$ (notice that it is contained in $T_{P_{\lambda}})$ is sent to the $i$ under the map $\kappa$.
We call $(P_{\lambda},\le,\kappa,I)$ the \emph{$I$-colored d-complete poset} for the minuscule weight $\lambda$.

\begin{prop}[{\cite[Proposition 9.1]{Proc2}}]\label{color filter,orbit}
Keep the notation and setting in Proposition \ref{iso filter,orbit}.
Let $(P_{\lambda},\le,\kappa,I)$ be the $I$-colored d-complete poset.
There exists a unique order isomorphism $f:(W\lambda,\le_s)\overset{\sim}{\rightarrow} (\mathcal{F}(P_{\lambda}),\subseteq)$ such that for $\mu\in W\lambda$ and $\ i\in I$, there exists the cover relation $\mu\rightarrow s_i\mu$ in $W\lambda$ if and only if $f(s_i(\mu))\setminus f(\mu)$ consists of one element having the color $i$.
\end{prop}

\begin{exam}\label{exam color filter,orbit}
Let $\Lieg$ be of type $A_5$, and $\lambda=\Lambda_{2}$; in this case, the corresponding (connected, self-dual) d-complete poset $P_{\Lambda_2}$ is $Y_{2,4}$.
Let $(P_{\Lambda_2},\le,\kappa,I)$ be the $I$-colored d-complete poset, with the coloring $\kappa$ as in Figure \ref{fig:c-d-comp}.
The Hasse diagrams of $(W\Lambda_2,\le_s)$ and $(\mathcal{F}(P_{\Lambda_2}),\subseteq)$ are given in Figure \ref{fig:color filter,orbit} below:
\end{exam}

\begin{figure}[h]
\centering
\vspace{10pt}
{\unitlength 0.1in%
\begin{picture}(21.3700,32.0000)(-4.3700,-33.2700)%
\put(2.0000,-34.0000){\makebox(0,0){\scriptsize$(0,1,0,0,0)$}}%
\put(6.0000,-30.0000){\makebox(0,0){\scriptsize$(1,-1,1,0,0)$}}%
\put(10.0000,-26.0000){\makebox(0,0){\scriptsize$(1,0,-1,1,0)$}}%
\put(14.0000,-22.0000){\makebox(0,0){\scriptsize$(1,0,0,-1,1)$}}%
\put(18.0000,-18.0000){\makebox(0,0){\scriptsize$(1,0,0,0,-1)$}}%
\put(14.0000,-14.0000){\makebox(0,0){\scriptsize$(-1,1,0,0,-1)$}}%
\put(10.0000,-10.0000){\makebox(0,0){\scriptsize$(0,-1,1,0,-1)$}}%
\put(6.0000,-6.0000){\makebox(0,0){\scriptsize$(0,0,-1,1,-1)$}}%
\put(2.0000,-2.0000){\makebox(0,0){\scriptsize$(0,0,0,-1,0)$}}%
\put(2.0000,-10.0000){\makebox(0,0){\scriptsize$(0,0,-1,0,1)$}}%
\put(6.0000,-14.0000){\makebox(0,0){\scriptsize$(0,-1,1,-1,1)$}}%
\put(10.0000,-18.0000){\makebox(0,0){\scriptsize$(-1,1,0,-1,1)$}}%
\put(2.0000,-18.0000){\makebox(0,0){\scriptsize$(0,-1,0,1,0)$}}%
\put(6.0000,-22.0000){\makebox(0,0){\scriptsize$(-1,1,-1,1,0)$}}%
\put(2.0000,-26.0000){\makebox(0,0){\scriptsize$(-1,0,1,0,0)$}}%
%
\special{pn 8}%
\special{pa 300 300}%
\special{pa 500 500}%
\special{fp}%
\special{pa 700 700}%
\special{pa 900 900}%
\special{fp}%
\special{pa 1100 1100}%
\special{pa 1300 1300}%
\special{fp}%
\special{pa 1500 1500}%
\special{pa 1700 1700}%
\special{fp}%
\special{pa 1700 1900}%
\special{pa 1500 2100}%
\special{fp}%
\special{pa 1300 2300}%
\special{pa 1100 2500}%
\special{fp}%
\special{pa 900 2700}%
\special{pa 700 2900}%
\special{fp}%
\special{pa 500 3100}%
\special{pa 300 3300}%
\special{fp}%
\special{pa 500 2900}%
\special{pa 300 2700}%
\special{fp}%
\special{pa 300 2500}%
\special{pa 500 2300}%
\special{fp}%
\special{pa 500 2100}%
\special{pa 300 1900}%
\special{fp}%
\special{pa 300 1700}%
\special{pa 500 1500}%
\special{fp}%
\special{pa 500 1300}%
\special{pa 300 1100}%
\special{fp}%
\special{pa 300 900}%
\special{pa 500 700}%
\special{fp}%
\special{pa 700 1300}%
\special{pa 900 1100}%
\special{fp}%
\special{pa 700 1500}%
\special{pa 900 1700}%
\special{fp}%
\special{pa 1100 1700}%
\special{pa 1300 1500}%
\special{fp}%
\special{pa 1100 1900}%
\special{pa 1300 2100}%
\special{fp}%
\special{pa 900 1900}%
\special{pa 700 2100}%
\special{fp}%
\special{pa 700 2300}%
\special{pa 900 2500}%
\special{fp}%
\put(4.5000,-3.5000){\makebox(0,0){$s_4$}}%
\put(3.5000,-15.5000){\makebox(0,0){$s_4$}}%
\put(7.5000,-19.5000){\makebox(0,0){$s_4$}}%
\put(11.5000,-23.5000){\makebox(0,0){$s_4$}}%
\put(8.5000,-7.5000){\makebox(0,0){$s_3$}}%
\put(4.5000,-11.5000){\makebox(0,0){$s_3$}}%
\put(3.5000,-23.5000){\makebox(0,0){$s_3$}}%
\put(7.5000,-27.5000){\makebox(0,0){$s_3$}}%
\put(12.5000,-11.5000){\makebox(0,0){$s_2$}}%
\put(8.5000,-15.5000){\makebox(0,0){$s_2$}}%
\put(4.5000,-19.5000){\makebox(0,0){$s_2$}}%
\put(3.5000,-31.5000){\makebox(0,0){$s_2$}}%
\put(16.5000,-15.5000){\makebox(0,0){$s_1$}}%
\put(12.5000,-19.5000){\makebox(0,0){$s_1$}}%
\put(8.5000,-23.5000){\makebox(0,0){$s_1$}}%
\put(4.5000,-27.5000){\makebox(0,0){$s_1$}}%
\put(3.5000,-7.5000){\makebox(0,0){$s_5$}}%
\put(7.5000,-11.5000){\makebox(0,0){$s_5$}}%
\put(11.5000,-15.5000){\makebox(0,0){$s_5$}}%
\put(15.5000,-19.5000){\makebox(0,0){$s_5$}}%
\end{picture}}%
\qquad\qquad
\input{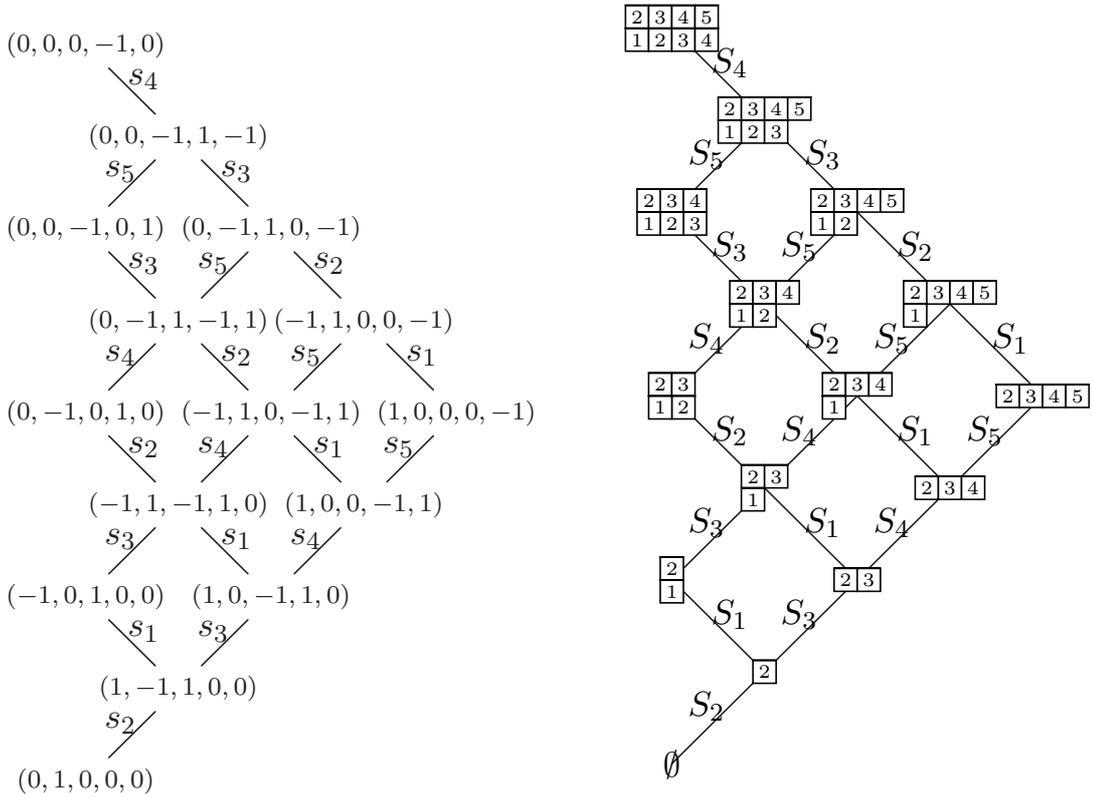}
\vspace{10pt}
\caption{$(W\Lambda_2,\le_s)$ and $(\mathcal{F}(P_{\Lambda_2}),\subseteq)$ of type $A_5$}\label{fig:color filter,orbit}
\end{figure}

The next corollary follows from Remark \ref{S-order,inclusion} and Proposition \ref{color filter,orbit}.

\begin{coro}\label{action filter,orbit}
Assume that $\Lieg$ is simply-laced.
Let $\lambda$ be a minuscule weight of $\Lieg$, and let $(P_{\lambda},\le)$ be the d-complete poset such that $(\mathcal{F}(P_{\lambda}),\subseteq)$ is isomorphic to $(W\lambda,\le_s)$ (see Proposition \ref{iso filter,orbit}).
Let $(P_{\lambda},\le,\kappa,I)$ be the $I$-colored d-complete poset, and let $f:(W\lambda,\le_s)\overset{\sim}{\rightarrow} (\mathcal{F}(P_{\lambda}),\subseteq)$ be the order isomorphism in Proposition \ref{color filter,orbit}.
For $\mu\in W\lambda$ and $i\in I $,
$$f(s_i(\mu))=S_i(f(\mu)).$$
\end{coro}

For $F \in \mathcal{F}(P_{\lambda})$ and $i \in I$, we define $c_i(F) := \#\{x\in F \mid \kappa(x) = i\}$.
Because $\lambda$ is minuscule, we see that if there exists the cover relation $\mu\rightarrow s_i(\mu)$ in $W\lambda$, then $\mu(h_i) = 1$ and $s_i(\mu) = \mu - \alpha_i$.
Hence we have the next corollary.

\begin{coro}\label{invf}
For $\mu \in W\lambda$ and $F = f(\mu)$,
$$\mu = \sum_{i\in I} (\#(S_i(F)) - \#(F))\Lambda_i = \lambda - \sum_{i\in I} c_i(F)\alpha_i .$$ 
\end{coro}
For $F \in \mathcal{F}(P_{\lambda})$, we define 
$$g(F) := \sum_{i\in I} (\#(S_i(F)) - \#(F))\Lambda_i = \lambda - \sum_{i\in I} c_i(F)\alpha_i .$$
By Corollary \ref{invf}, $g:(\mathcal{F}(P_{\lambda}),\subseteq)\overset{\sim}{\rightarrow} (W\lambda,\le_s)$ is the inverse of $f$.

We will use the following proposition later.

\begin{prop}[{\cite[Proposition 8.6]{Proc2}}]\label{property I-c-d-comp}
Keep the notation and setting in Proposition \ref{iso filter,orbit}.
Let $(P_{\lambda},\le,\kappa,I)$ be the $I$-colored d-complete poset.
If there exists  the covering relation between $x,y\in P_{\lambda}$, then the  color $\kappa(x)$ of $x$ is adjacent to the color $\kappa(y)$ of $y$ in the Dynkin diagram of $\Lieg$.
\end{prop}


\section{Folding of a Lie algebra.}\label{Sec:folding}

We review the ``folding'' of a simply-laced finite-dimensional simple Lie algebra; for the details, see \cite[Sections 7.9 and 7.10]{Kac} and \cite[Section 9.5]{Cart} in example.

\begin{figure}[h]
\centering
\input{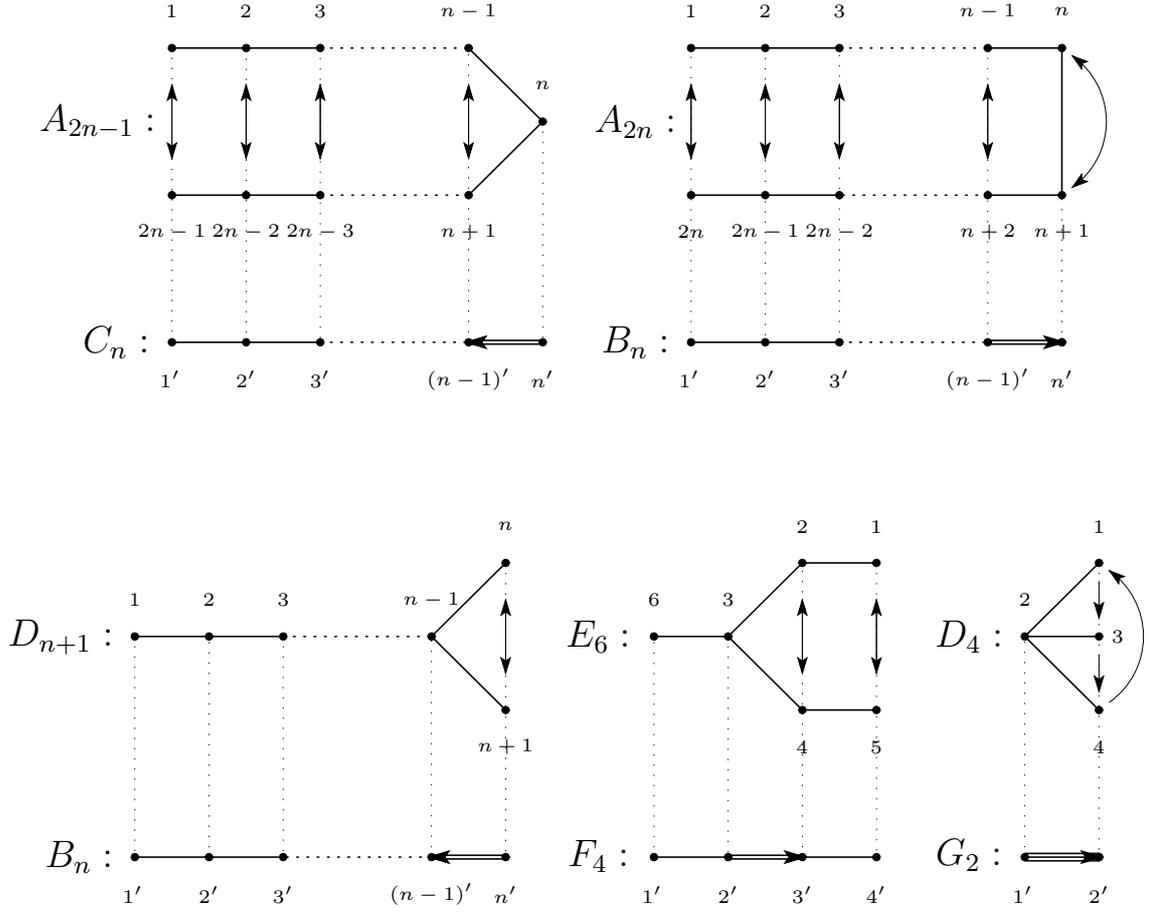}
\vspace{10pt}
\caption{The Dynkin diagram of $\Lieg$, its (non-trivial) graph automorphism $\sigma:I\rightarrow I$, and the Dynkin diagram of the fixed point subalgebra $\Lieg(0)$.}\label{fig:dynkin index2}
\end{figure}

Let $\Lieg$ be the finite-dimensional  simple Lie algebra of type $A_n,D_n$ or $E_6$; we use the notation in Section \ref{Sec:Lie}.
Let $\sigma$ be a non-trivial graph automorphism of the Dynkin diagram of $\Lieg$.
Denote by $\langle\sigma\rangle$ the cyclic group generated by $\sigma$ (in the group of permutations on $I$), and $J$ the set of $\langle\sigma\rangle$-orbits on $I$.
We say that $p\in J$ satisfies the \emph{orthogonality condition} if $a_{ij}=a_{ji}=0$ for all $i,j\in p$ with $i\not=j$; notice that $p\in J$ does not satisfy the orthogonality condition if and only if $\Lieg$ is of type $A_{2n}$ and $p=\{n,n+1\}$.\label{not orth}
It is known that the graph automorphism $\sigma$ induces a (unique) Lie algebra automorphsim of $\Lieg$ such that $\sigma(e_i)=e_{\sigma(i)},\sigma(f_i)=f_{\sigma(i)},\sigma(h_i)=h_{\sigma(i)}$ for $i\in I$; we set $\Lieg(0):=\{x\in\Lieg\,|\,\sigma(x)=x\}$.
For each $p\in J$, we define $H_p,E_p,F_p\in \Lieg(0)$ as follows:

\begin{enumerate}\renewcommand{\labelenumi}{(\arabic{enumi})}
\item If $p$ satisfies the orthogonality condition, then
$$H_p:=\sum_{i\in p}h_i,\qquad E_p:=\sum_{i\in p}e_i,\qquad F_p:=\sum_{i\in p}f_i.$$
\item If $p$ does not satisfy the orthogonality condition, then
$$H_p:=2\sum_{i\in p}h_i,\qquad E_p:=\sum_{i\in p}e_i,\qquad F_p:=2\sum_{i\in p}f_i.$$
\end{enumerate}

\begin{prop}[see, e.g., {\cite[Sections 7.9 and 7.10]{Kac}}]\label{lie fold}
The fixed point subalgebra $\Lieg(0)$ is generated by $\{H_p,E_p,F_p\}_{p\in J}$, and is isomorphic to a multiply-laced finite-dimensional simple Lie algebra; see Figure \ref{fig:dynkin index2} and Table \ref{tab:lie fold}.
\end{prop}

\begin{table}[h]
\begin{center}
\begin{tabular}{|c||c|c|c|c|c|}
\hline
type of $\Lieg$&$A_{2n}$&$A_{2n-1}$&$D_{n+1}$&$E_6$&$D_4$\\
\hline
order of $\sigma$&2&2&2&2&3\\
\hline
type of $\Lieg(0)$&$B_{n}$&$C_{n}$&$B_{n}$&$F_4$&$G_2$\\
\hline
\end{tabular}
\caption{$\Lieg$, $\sigma$, and $\Lieg(0)$. The vertices of the Dynkin diagram of $\Lieg(0)$ are ``numbered'' as Figure \ref{fig:dynkin index2}.}\label{tab:lie fold}
\end{center}
\end{table}

Let $\Lieh(0)$ be the subspace of $\Lieh$ spanned by $\{H_p\}_{p\in J}$, which is a Cartan subalgebra of $\Lieg(0)$.
Denote by $\rm res:\Lieh^*\rightarrow \Lieh(0)^*,\mu\mapsto \mu|_{\Lieh(0)}$, the restriction map, and set $\beta_p:=\text{res}( \alpha_i)\in \Lieh(0)^*$ for $p\in J$, where $i$ is an arbitrary element in the $\langle\sigma\rangle$-orbit $p$; note that $\beta_p$ is independent of the choice of $i\in p$.
The set of simple coroots and the set of simple roots of $\Lieg(0)$ are given by $\{H_p\}_{p\in J}$ and $\{\beta_p\}_{p\in J}$, respectively.
Denote by $\tilde{\Delta}_+\subset \Lieh(0)^*$ the set of positive roots of $\Lieg(0)$, and $\tilde{\Delta}_-\subset \Lieh(0)^*$ the set of negative roots of $\Lieg(0)$.
For $p\in J$, we define $\tilde{s}_p(\nu):=\nu-\nu(H_p) \beta_p$ for $\nu\in \Lieh(0)^*$.
Then, $\tilde{W}:=\langle \tilde{s}_p\mid p\in J\rangle$ is the Weyl group of $\Lieg(0)$.

For each $p\in J$, we define $\hat{s}_p\in W$ as follows:

\begin{enumerate}\renewcommand{\labelenumi}{(\arabic{enumi})}
\item  If $p$ satisfies the orthogonality condition, then
$$\hat{s}_p:=\prod_{k\in p}s_k.$$
\item If $p$ does not satisfy the orthogonality condition, that is, if $\Lieg$ is of type $A_{2n}$ and $p=\{n,n+1\}$ (see also page \pageref{not orth}), then
$$\hat{s}_p:= s_n s_{n+1} s_n = s_{n+1} s_n s_{n+1} .$$
\end{enumerate} 

\begin{lemm}\label{property weyl fold}
For $p\in J$, $\tilde{s}_p(\text{res}(\mu))=\text{res}(\hat{s}_p (\mu))$ for all $\mu\in\Lieh^*$.
\end{lemm}

\begin{proof}
If $p$ satisfies the orthogonality condition, then we compute
\begin{equation*}
\text{res}(\hat{s}_p (\mu))=\text{res}\left
(\mu-\sum_{i\in p}\mu(h_i)\alpha_i\right)=\text{res}(\mu)-\text{res}(\mu)(H_p)\beta_p=\tilde{s}_p(\text{res}(\mu)).
\end{equation*}
If $p$ does not satisfy the orthogonality condition, then we compute
\begin{equation*}\begin{split}
\text{res}(\hat{s}_p (\mu))&=\text{res}(\mu-\mu(h_n+h_{n+1})(\alpha_n+\alpha_{n+1}))\\
&=\text{res}(\mu)-\text{res}(\mu)(H_p)\beta_p=\tilde{s}_p(\text{res}(\mu)).
\end{split}\end{equation*}
\end{proof}

Since $\sigma$ acts on $\Lieh=\bigoplus_{i\in I}\mathbb{C}h_i$, $\sigma$ naturally acts also on $\Lieh^*$ by $(\sigma(\mu))(h)=\mu(\sigma^{-1}(h))$ for $\mu\in\Lieh^*$ and $h\in\Lieh$; we see that $\sigma(\Lambda_i)=\Lambda_{\sigma(i)},\sigma(\alpha_i)=\alpha_{\sigma(i)}$ for $i\in I$.
Notice that $\sigma s_i \sigma^{-1}=s_{\sigma(i)}$ for $i\in I$ in ${\it GL}(\Lieh^*)$.
Hence, $\sigma W\sigma^{-1}\subseteq W$.

\begin{prop}[{\cite[Proposition 9.17]{Cart}}]\label{hat W and tilde W}
Set $\hat{W}:=\{w\in W\mid \sigma w \sigma^{-1}=w\}$.
There is a group isomorphism from $\hat{W}$ onto $\tilde{W}$ such that $\hat{s}_p \mapsto \tilde{s}_p$.
Therefore $\hat{W}$ is the subgroup of $W$ generated by $\{\hat{s}_p\}_{p\in J}$.
\end{prop}

\begin{rema}\label{hat orbit}
Because $\tilde{W}$ and $\hat{W}$ are generated by $\{\tilde{s}_p\}_{p\in J}$ and $\{\hat{s}_p\}_{p\in J}$, we see by Lemma \ref{property weyl fold} that $\text{res}(\hat{W}\lambda)=\tilde{W}\text{res}(\lambda)$ for every (dominant) integral weight $\lambda$.
\end{rema}

Let $\tilde\Lambda_p\in\Lieh(0)^*(p\in J)$ be the fundamental weights of $\Lieg(0)$.
We can easily show the following lemma.

\begin{lemm}\label{weight fold}
Let $p\in J$, and $i\in p$.
\begin{enumerate}\renewcommand{\labelenumi}{(\arabic{enumi})}
\item If $p$ satisfies the orthogonality condition, then $\text{res}(\Lambda_i)=\tilde\Lambda_p$.
\item If $p$ does not satisfy the orthogonality condition, then $\text{res}(\Lambda_i)=2\tilde\Lambda_p$.
\end{enumerate}
\end{lemm}

\begin{lemm}\label{res injection}
Let $\lambda$ be a dominant integral weight of $\Lieg$, and let $\mu_1,\mu_2\in \hat{W}\lambda$.
If $\text{res}(\mu_1)=\text{res}(\mu_2)$, then $\mu_1=\mu_2$.
Therefore the map $\text{res}|_{\hat{W}\lambda}: \hat{W}\lambda\rightarrow\tilde{W}\text{res}(\lambda)$ is bijective (see Remark \ref{hat orbit}).
\end{lemm}

\begin{proof}
For each $i=1,2$, let $\hat{w_i}\in\hat{W}$ be such that $\mu_i=\hat{w_i}\lambda$, and let $\tilde{w_i}\in\tilde{W}$ be such that $\text{res}\circ\hat{w_i}=\tilde{w_i}\circ\text{res}$ (see Lemma \ref{property weyl fold}).
We have $\tilde{w_1}\text{res}(\lambda)=\text{res}(\hat{w_1}\lambda)=\text{res}(\mu_1)=\text{res}(\mu_2)=\text{res}(\hat{w_2}\lambda)=\tilde{w_2}\text{res}(\lambda)$.
Since $\text{res}(\lambda)$ is a dominant integral weight for $\Lieg(0)$ by Lemma \ref{weight fold}, it follows that $\tilde{w_1}^{-1}\tilde{w_2}\in\langle\tilde{s}_p \mid (\text{res}(\lambda))(H_p)=0\rangle$, and hence $\hat{w_1}^{-1}\hat{w_2}\in\langle\hat{s}_p \mid (\text{res}(\lambda))(H_p)=0\rangle$.
Observe that $(\text{res}(\lambda))(H_p)=0$ if and only if $\lambda(h_i)=0$ for all $i\in p$.
Thus we obtain $\hat{w_1}^{-1}\hat{w_2}(\lambda)=\lambda$, and hence $\mu_1=\hat{w_1}\lambda=\hat{w_2}\lambda=\mu_2$, as desired.
\end{proof}

Notice that $\sigma$ preserves $\Delta$ and $\Delta_+, \Delta_-$.

\begin{lemm}\label{sign coincide}
Let $\lambda$ be a dominant integral weight of $\Lieg$.
\begin{enumerate}\renewcommand{\labelenumi}{(\arabic{enumi})}
\item[(1)] For each $\mu\in \hat{W}\lambda$ and $p\in J$, either $\mu(h_{i})\ge0$ for all $i\in p$ or $\mu(h_{i})\le0$ for all $i\in p$.
\item[(2)] For each $\mu\in \hat{W}\lambda$ and $p\in J$, if $\mu(h_{i})>0$ (resp., $\mu(h_{i})<0$) for some $i\in p$, then $\mu <_w \hat{s}_p(\mu)$ (resp., $\mu >_w \hat{s}_p(\mu)$).
\end{enumerate}
\end{lemm}

\begin{proof}
(1) Let $w\in\hat{W}$ be such that $\mu=w\lambda$.
Because $\mu(h_i)=(w\lambda)(h_i)=\lambda(w^{-1}h_i)=\lambda((w^{-1}\alpha_i)^\vee)$, and because $\lambda$ is a dominant integral weight, it suffices to show that either $w^{-1}\alpha_i\in\Delta_+$ for all $i\in p$ or  $w^{-1}\alpha_i\in\Delta_-$ for all $i\in p$.
If $w^{-1}\alpha_i\in\Delta_+$ (resp., $w^{-1}\alpha_i\in\Delta_-$) for some $i\in p$, then $w^{-1}\alpha_{\sigma(i)}= w^{-1}\sigma\alpha_{i}=\sigma w^{-1}\alpha_{i}\in \Delta_+$ (resp., $\in w^{-1}\alpha_i\in\Delta_-$).
Since $p$ is a $\langle\sigma\rangle$-orbit, the assertion above follows.

(2) We give a proof only for the case that $\mu(h_{i})>0$ for some $i\in p$, and $\#p = 2$; the proofs for the other cases are similar.
Since $\mu(h_{i})>0$, it follows that $\mu <_w s_i(\mu)$.
If $p=\{i,j\}$, then we see by part (1) that $\mu(h_{j})\ge0$.
Assume that $p$ satisfies the orthogonality condition.
Then,
\begin{eqnarray*}
s_js_i(\mu) &=& s_j(\mu-\mu(h_i)\alpha_i) = s_j(\mu)-\mu(h_i)s_j(\alpha_i)\\
&=& \mu - \mu(h_j)\alpha_j-\mu(h_i)\alpha_i = s_i(\mu) - \mu(h_j)\alpha_j \ge_w s_i(\mu).
\end{eqnarray*}
Thus we obtain $\mu <_w s_i(\mu) \le_w s_js_i(\mu) = \hat{s}_p(\mu)$, as desired.
Assume that $p$ does not satisfy the orthogonality condition.
Then,
\begin{eqnarray*}
s_js_i(\mu) &=& s_j(\mu-\mu(h_i)\alpha_i) = s_j(\mu)-\mu(h_i)s_j(\alpha_i)\\
&=& \mu - \mu(h_j)\alpha_j-\mu(h_i)(\alpha_i+\alpha_j) = \mu - \mu(h_i)\alpha_i - \mu(h_i)\alpha_j - \mu(h_j)\alpha_j\\
&=& s_i(\mu) - \mu(h_i)\alpha_j - \mu(h_j)\alpha_j >_w s_i(\mu),
\end{eqnarray*}
\begin{eqnarray*}
s_is_js_i(\mu) &=& s_i(s_i(\mu) - \mu(h_i+h_j)\alpha_j) = \mu - s_i(\mu(h_i+h_j)\alpha_j)\\
&=& \mu - \mu(h_i+h_j)(\alpha_i+\alpha_j) = \mu - \mu(h_i)\alpha_i - \mu(h_j)\alpha_i - \mu(h_i+h_j)\alpha_j\\
&\ge_w& \mu - \mu(h_i)\alpha_i - \mu(h_i+h_j)\alpha_j = s_js_i(\mu).
\end{eqnarray*}
Thus we obtain $\mu <_w s_i(\mu) <_w s_js_i(\mu) \le_w s_is_js_i(\mu) = \hat{s}_p(\mu)$, as desired.
\end{proof}

\begin{defi}\label{ht}
We set $Q_+:= \sum_{i\in I}\mathbb{Z}_{\ge0}\alpha_i$.
For $\nu = \sum_{i\in I}m_i\alpha_i \in Q_+$, we define the height $\text{ht}(\nu)$ of $\nu$ by $\text{ht}(\nu):= \sum_{i\in I}m_i$.
Similarly, we set $\tilde{Q}_+:= \sum_{p\in J}\mathbb{Z}_{\ge0}\beta_p$.
For $\xi = \sum_{p\in J}n_p\beta_p \in \tilde{Q}_+$, we define the height $\text{ht}(\xi)$ of $\xi$ by $\text{ht}(\xi):= \sum_{p\in J}n_p$.
\end{defi}

\begin{lemm}\label{bruhat fold}
Let $\lambda$ be a dominant integral weight, and $\mu_1,\mu_2\in \hat{W}\lambda$.
Then, $\mu_1 \le_s \mu_2$ if and only if $\text{res}(\mu_1) \le_s \text{res}(\mu_2)$.
\end{lemm}
\begin{proof}
First, we show the ``if'' part.
We see that $\text{res}(\lambda) - \text{res}(\mu_2) \in \tilde{Q}_+$ since $\text{res}(\lambda)$ is dominant and $\text{res}(\mu_2) \in \tilde{W}\text{res}(\lambda)$.
We show the assertion by induction on $\tilde{h}:=\text{ht}(\text{res}(\lambda) - \text{res}(\mu_2))$.
If $\tilde{h} = 0$, then $\text{res}(\mu_2) = \text{res}(\lambda)$.
Because $\text{res}(\lambda) - \text{res}(\mu_1) \in \tilde{Q}_+$, and because $\text{res}(\mu_1) - \text{res}(\lambda) = \text{res}(\mu_1) - \text{res}(\mu_2) \in \tilde{Q}_+$ by the definition of $\le_s$ on $\tilde{W}\text{res}(\lambda)$, we get $\text{res}(\mu_1) = \text{res}(\lambda)$.
Now, for $i = 1,2$, we see that $\lambda - \mu_i \in Q_+$.
Since $\text{res}(\lambda - \mu_i) = \text{res}(\lambda) - \text{res}(\mu_i) = \text{res}(\lambda) - \text{res}(\lambda) = 0$, we deduce that $\lambda = \mu_i$.
Thus, we obtain $\mu_1 = \lambda \le_s \lambda = \mu_2$.

Assume that $\tilde{h} > 0$.
In this case, there exists $p \in J$ such that $\text{res}(\mu_2)(H_p) < 0$, because $\text{res}(\lambda)$ is a unique dominant integral weight in $\tilde{W}\text{res}(\lambda)$; note that $\text{ht}(\text{res}(\lambda) - \tilde{s}_p\text{res}(\mu_2)) < \tilde{h}$.
Here, we give a proof only for the case that $p=\{i,j\}$ with $i\not=j$, and $p$ satisfies the orthogonality condition; the proofs for the other cases are similar.
If $\text{res}(\mu_1)(H_p) \ge 0$, then we get $\text{res}(\mu_1) \le_s  \tilde{s}_p\text{res}(\mu_2)$ by Proposition \ref{strong-slide} (1).
By the induction hypothesis, it follows that $\mu_1 \le_s \hat{s}_p(\mu_2)$.
Because $\mu_2(h_i+h_j) = \text{res}(\mu_2)(H_p) < 0$, we see by Lemma \ref{sign coincide} that $\hat{s}_p(\mu_2) = s_j s_i(\mu_2) \le_s s_i (\mu_2) \le_s \mu_2$.
Thus we obtain $\mu_1 \le_s \mu_2$.
If $\text{res}(\mu_1)(H_p) \le 0$, then we get $\tilde{s}_p\text{res}(\mu_1) \le_s  \tilde{s}_p\text{res}(\mu_2)$ by Proposition \ref{strong-slide} (3).
By the induction hypothesis, it follows that $\hat{s}_p(\mu_1) \le_s \hat{s}_p(\mu_2)$.
Similarly to the case above, we deduce that $\hat{s}_p(\mu_k)(h_i) \le 0$ and $s_i\hat{s}_p (\mu_k)(h_j) \le 0$ for $k = 1,2$.
By Proposition \ref{strong-slide} (4), we obtain $s_i\hat{s}_p (\mu_1) \le_s s_i\hat{s}_p (\mu_2)$, and then $\mu_1= s_js_i\hat{s}_p (\mu_1) \le_s s_js_i\hat{s}_p (\mu_2) = \mu_2$, as desired.

Next, we show the ``only if'' part by the induction on $h := \text{ht}(\lambda - \mu_2)$.
If $h = 0$, then we see by the same argument as above that $\mu_1 = \mu_2 = \lambda$.
Hence, $\text{res}(\mu_1) \le_s \text{res}(\mu_2)$.
Assume that $h > 0$.
Then there exists $i \in I$ such that $\mu_2(h_i) < 0$.
Let $p \in J$ be such that $i \in p$.
Here, we give a proof only for the case that $p=\{i,j\}$ with $i\not=j$, and $p$ satisfies the orthogonality condition; the proofs for the other cases are similar.
By Lemma \ref{sign coincide}, $\mu_2(h_j) \le 0$ and $\hat{s}_p(\mu_2) = s_js_i(\mu_2)\le_s s_i(\mu_2) <_s \mu_2$; note that $\text{ht}(\lambda - \hat{s}_p(\mu_2)) < h$.
Assume that $\mu_1(h_i)\ge 0$.
It follows from Proposition \ref{strong-slide} (1) that $\mu_1 \le_s  s_i(\mu_2)$.
Also, we see by Lemma \ref{sign coincide} (1) that $\mu_1(h_j)\ge 0$.
By Proposition \ref{strong-slide} (1), we get $\mu_1 \le_s s_js_i(\mu_2) = \hat{s}_p(\mu_2)$.
By the induction hypothesis, it follows that $\text{res}(\mu_1) \le_s \tilde{s}_p\text{res}(\mu_2)$.
Because $\text{res}(\mu_2)(H_p) = \mu_2(h_i+h_j) < 0$, we have $\tilde{s}_p\text{res}(\mu_2) \le_s \text{res}(\mu_2)$, and hence $\text{res}(\mu_1) \le_s \text{res}(\mu_2)$.
Assume that $\mu_1(h_i)\le 0$.
It follows from Proposition \ref{strong-slide} (3) that $s_i(\mu_1) \le_s  s_i(\mu_2)$.
Also, we see by Lemma \ref{sign coincide} (2) that $(s_i(\mu_1))(h_j)\le 0$.
By Proposition \ref{strong-slide} (3), we get $\hat{s}_p(\mu_1) = s_js_i(\mu_1) \le_s s_js_i(\mu_2) = \hat{s}_p(\mu_2)$.
By the induction hypothesis, it follows that $\tilde{s}_p\text{res}(\mu_1) \le_s \tilde{s}_p\text{res}(\mu_2)$.
Because $\text{res}(\mu_1)(H_p) \le 0$ and $\text{res}(\mu_2)(H_p) \le 0$, we obtain $\text{res}(\mu_1) \le_s \text{res}(\mu_2)$ by Proposition \ref{strong-slide} (4), as desired.
\end{proof}


\section{$J$-colored d-complete poset.}\label{Sec:J-color}


Let $\Lieg$ be a simply-laced finite-dimensional Lie algebra, and let $\sigma$ be a non-trivial graph automorphsim of the Dynkin diagram of $\Lieg$ (see Figure \ref{fig:dynkin index2}).
Let $\lambda$ be a minuscule weight of $\Lieg$.
Recall from Proposition \ref{iso filter,orbit} that there exists a connected self-dual d-complete poset $(P_{\lambda},\le)$ such that $(W\lambda,\le_s)$ and $(\mathcal{F}(P_{\lambda}),\subseteq)$ are isomorphic.
Let $(P_{\lambda},\le,\kappa,I)$ be the $I$-colored d-complete poset (see the comment after Proposition \ref{iso filter,orbit}).
By Proposition \ref{color filter,orbit} and Corollary \ref{action filter,orbit}, there exists a unique order isomorphism $f:(W\lambda,\le_s)\overset{\sim}{\rightarrow} (\mathcal{F}(P_{\lambda}),\subseteq)$ such that $f(s_i(\mu))=S_i(f(\mu))$ for all $\mu\in W\lambda$ and $i\in I $.
Because the map $\text{res}|_{\hat{W}\lambda}:\hat{W}\lambda\rightarrow\tilde{W}\text{res}(\lambda)$ is bijective (see Lemma \ref{res injection}), we can define a map $\tilde{f}:\tilde{W}\text{res}(\lambda)\rightarrow \mathcal{F}(P_{\lambda})$ by the following commutative diagram (\ref{commutative f-tilde}):

\begin{eqnarray}\label{commutative f-tilde}
\begin{split}
{\unitlength 0.1in%
\begin{picture}(15.8500,12.0000)(-3.8500,-15.3500)%
\put(14.0000,-10.0000){\makebox(0,0){$(\mathcal{F}(P_\lambda),\subseteq)$}}%
\put(4.0000,-10.0000){\makebox(0,0){$\hat{W}\lambda$}}%
\put(4.0000,-4.0000){\makebox(0,0){$W\lambda$}}%
\put(4.0000,-16.0000){\makebox(0,0){$\tilde{W}\text{res}(\lambda)$}}%
\put(4.0000,-7.0000){\rotatebox{90.0000}{\makebox(0,0){$\subseteq$}}}%
%
\special{pn 4}%
\special{pa 400 1200}%
\special{pa 400 1400}%
\special{fp}%
\special{sh 1}%
\special{pa 400 1400}%
\special{pa 420 1333}%
\special{pa 400 1347}%
\special{pa 380 1333}%
\special{pa 400 1400}%
\special{fp}%
%
\special{pn 4}%
\special{pa 720 1520}%
\special{pa 1200 1200}%
\special{fp}%
\special{sh 1}%
\special{pa 1200 1200}%
\special{pa 1133 1220}%
\special{pa 1156 1230}%
\special{pa 1156 1254}%
\special{pa 1200 1200}%
\special{fp}%
%
\special{pn 4}%
\special{pa 600 400}%
\special{pa 1200 800}%
\special{fp}%
\special{sh 1}%
\special{pa 1200 800}%
\special{pa 1156 746}%
\special{pa 1156 770}%
\special{pa 1133 780}%
\special{pa 1200 800}%
\special{fp}%
%
\special{pn 4}%
\special{pa 600 1000}%
\special{pa 1000 1000}%
\special{fp}%
\special{sh 1}%
\special{pa 1000 1000}%
\special{pa 933 980}%
\special{pa 947 1000}%
\special{pa 933 1020}%
\special{pa 1000 1000}%
\special{fp}%
\put(9.0000,-5.0000){\makebox(0,0){$f$}}%
\put(9.0000,-15.0000){\makebox(0,0){$\tilde{f}$}}%
\put(2.0000,-13.0000){\makebox(0,0){res}}%
\put(8.0000,-8.5000){\makebox(0,0){$f|_{\hat{W}\lambda}$}}%
\put(7.5000,-12.5000){\makebox(0,0){$\circlearrowleft$}}%
\end{picture}}%
\end{split}
\end{eqnarray}
\vspace{5pt}

\noindent We define $\tilde{\mathcal{F}}(P_{\lambda}):=\text{Im}(\tilde{f})\subseteq \mathcal{F}(P_{\lambda})$ (see also (\ref{tilde-F}) below).

\begin{defi}\label{j-c-d-comp}
Keep the setting above.
We define a map $\tilde{\kappa}:P_{\lambda}\rightarrow J$ to be the composition of $\kappa:P_\lambda\rightarrow I$ and the canonical projection $I\twoheadrightarrow J$.
We call the colored poset $(P_{\lambda},\le,\tilde{\kappa},J)$ the \emph{$J$-colored d-complete poset} corresponding to $\Lieg(0)$ and $\text{res}(\lambda)$.
\end{defi}

For $F \in \mathcal{F}(P_{\lambda})$ and $p \in J$, we define $\tilde{c}_p(F) := \#\{x\in F \mid \tilde{\kappa}(x) = p\}$.
By Collorary \ref{invf}, it follows that for $\mu \in \hat{W}\lambda$ and $F = f(\mu)$,
$$\text{res}(\mu) = \text{res}(\lambda) - \sum_{p\in J}\Biggl(\sum_{i\in p} c_i(F)\Biggr)\beta_p = \text{res}(\lambda) - \sum_{p\in J} \tilde{c}_p(F)\beta_p.$$
We define $\tilde{g}: \tilde{\mathcal{F}}(P_{\lambda}) \to \tilde{W}\text{res}(\lambda)$ by 
$$\tilde{g}(F) := \text{res}(\lambda) - \sum_{p\in J} \tilde{c}_p(F)\beta_p$$
for $F \in \tilde{\mathcal{F}}(P_{\lambda})$.
It can be easily checked that $\tilde{g}$ is the inverse of $\tilde{f}$.

Denote by $\tilde{A}_p,\tilde{R}_p,\tilde{S}_p:\mathcal{F}(P_\lambda)\rightarrow \mathcal{F}(P_\lambda)\ (p\in J)$ the maps in Definition \ref{filter involution} for the $J$-colored d-complete poset $(P_{\lambda},\le,\tilde{\kappa},J)$.
Also, we define the order $\tilde{\unlhd}$ on $\mathcal{F}(P_\lambda)$ in exactly the same way as Definition \ref{S-order}.
Namely, for $F,F'\in \mathcal{F}(P_\lambda)$, $F\ \tilde{\unlhd}\ F'$ if there exists a sequence of order filters $F=F_0,F_1,\ldots,F_{n-1},F_n=F'$ in $\mathcal{F}(P_\lambda)$ such that for each $i\in \{0,1,\ldots ,n-1\}$, there exists $p_i\in J$ such that $\tilde{S}_{p_i}(F_i)=F_{i+1}\supset F_i$.

\begin{theo}[main result]\label{iso fold filter,orbit}
Keep the notation and setting above.
\begin{enumerate}\renewcommand{\labelenumi}{(\arabic{enumi})}
\item[(1)]\ The poset $(\tilde{W}\text{res}(\lambda),\le_w)$ is isomorphic to the poset $(\tilde{\mathcal{F}}(P_{\lambda}),\tilde{\unlhd})$ under the map $\tilde{f}: \tilde{W}\text{res}(\lambda) \rightarrow \tilde{\mathcal{F}}(P_{\lambda})$.
\item[(2)]\ The poset $(\tilde{W}\text{res}(\lambda),\le_s)$ is isomorphic to the poset $(\tilde{\mathcal{F}}(P_{\lambda}),\subseteq)$ under the map $\tilde{f}: \tilde{W}\text{res}(\lambda) \rightarrow \tilde{\mathcal{F}}(P_{\lambda})$.
\end{enumerate}
\end{theo}

\renewcommand{\arraystretch}{1.2}
\begin{table}[h]
\begin{center}
\begin{tabular}{|c|c|c|c|c|}
\hline
$\Lieg$&$\lambda$&$\Lieg(0)$&$\text{res}(\lambda)$&$P_{\lambda}$\\
\hline
$A_{2n-1}$&$\Lambda_1,\ldots,\Lambda_n$&$C_n$&$\tilde\Lambda_{1'},\ldots,\tilde\Lambda_{n'}$&Shape\\
\hline
$A_{2n}$&$\Lambda_1,\ldots,\Lambda_{n-1},\Lambda_n$&$B_n$&$\tilde\Lambda_{1'},\ldots,\tilde\Lambda_{(n-1)'},2\tilde\Lambda_{n'}$&Shape\\
\hline
$D_{n+1}$&$\Lambda_1$&$B_n$&$\tilde\Lambda_{1'}$&Inset\\
\hline
$D_{n+1}$&$\Lambda_n$&$B_n$&$\tilde\Lambda_{n'}$&Shifted Shape\\
\hline
$E_6$&$\Lambda_1$&$F_4$&$\tilde\Lambda_{4'}$&Swivel\\
\hline
$D_4$&$\Lambda_1$&$G_2$&$\tilde\Lambda_{2'}$&Shifted Shape\\
\hline
\end{tabular}
\caption{Correspondence between $\Lieg$, $\lambda$, $\Lieg(0)$, $\text{res}(\lambda)$, and $P$}\label{tab:iso fold filter,orbit}
\end{center}
\end{table}

\begin{exam}\label{exam iso fold filter,orbit}
Let $\Lieg$ be of type $A_5$, and $\lambda=\Lambda_{2}$.
Recall from Example \ref{exam color filter,orbit} that the corresponding (connected, self-dual) d-complete poset $P_{\Lambda_2}$ is $Y_{2,4}$, and the $I$-colored d-complete poset $(P_{\Lambda_2},\le,\kappa,I)$ is the left diagram in Figure \ref{fig:c-d-comp}.
In this case, $\Lieg(0)$ is of type $C_3$, and $\text{res}(\Lambda_2)=\tilde{\Lambda}_{2'}$.
The $J$-colored d-complete poset $(P_{\Lambda_{2}},\le,\tilde{\kappa},J)$ is below.
The Hasse diagrams of $(\tilde{W}\tilde{\Lambda}_{2'},\le_w)$ and $(\tilde{\mathcal{F}}(P_{\Lambda_{2}}),\tilde{\unlhd})$ are given in Figure \ref{fig:fold color filter,orbit}.
\end{exam}

\begin{figure}[h]
\centering
\vspace{10pt}
{\unitlength 0.1in%
\begin{picture}(17.3700,32.0000)(-3.3700,-33.2700)%
%
\special{pn 8}%
\special{pa 300 300}%
\special{pa 500 500}%
\special{fp}%
\special{pa 700 700}%
\special{pa 900 900}%
\special{fp}%
\special{pa 1100 1100}%
\special{pa 1300 1300}%
\special{fp}%
\special{pa 1400 1500}%
\special{pa 1400 2100}%
\special{fp}%
\special{pa 1300 2300}%
\special{pa 1100 2500}%
\special{fp}%
\special{pa 900 2700}%
\special{pa 700 2900}%
\special{fp}%
\special{pa 500 3100}%
\special{pa 300 3300}%
\special{fp}%
\special{pa 500 2900}%
\special{pa 300 2700}%
\special{fp}%
\special{pa 300 2500}%
\special{pa 500 2300}%
\special{fp}%
\special{pa 700 2300}%
\special{pa 900 2500}%
\special{fp}%
\special{pa 600 2100}%
\special{pa 600 1500}%
\special{fp}%
\special{pa 700 1300}%
\special{pa 900 1100}%
\special{fp}%
\special{pa 500 1300}%
\special{pa 300 1100}%
\special{fp}%
\special{pa 300 900}%
\special{pa 500 700}%
\special{fp}%
\put(2.0000,-34.0000){\makebox(0,0){\scriptsize$(0,1,0)$}}%
\put(6.0000,-30.0000){\makebox(0,0){\scriptsize$(1,-1,1)$}}%
\put(10.0000,-26.0000){\makebox(0,0){\scriptsize$(1,1,-1)$}}%
\put(14.0000,-22.0000){\makebox(0,0){\scriptsize$(2,-1,0)$}}%
\put(14.0000,-14.0000){\makebox(0,0){\scriptsize$(-2,1,0)$}}%
\put(10.0000,-10.0000){\makebox(0,0){\scriptsize$(-1,-1,1)$}}%
\put(6.0000,-6.0000){\makebox(0,0){\scriptsize$(-1,1,-1)$}}%
\put(2.0000,-2.0000){\makebox(0,0){\scriptsize$(0,-1,0)$}}%
\put(2.0000,-10.0000){\makebox(0,0){\scriptsize$(1,0,-1)$}}%
\put(6.0000,-14.0000){\makebox(0,0){\scriptsize$(1,-2,1)$}}%
\put(6.0000,-22.0000){\makebox(0,0){\scriptsize$(-1,2,-1)$}}%
\put(2.0000,-26.0000){\makebox(0,0){\scriptsize$(-1,0,1)$}}%
\put(5.0000,-3.5000){\makebox(0,0){$\tilde{s}_{2'}$}}%
\put(7.0000,-18.0000){\makebox(0,0){$\tilde{s}_{2'}$}}%
\put(11.0000,-23.5000){\makebox(0,0){$\tilde{s}_{2'}$}}%
\put(9.0000,-7.5000){\makebox(0,0){$\tilde{s}_{3'}$}}%
\put(5.0000,-11.5000){\makebox(0,0){$\tilde{s}_{3'}$}}%
\put(3.0000,-23.5000){\makebox(0,0){$\tilde{s}_{3'}$}}%
\put(7.0000,-27.5000){\makebox(0,0){$\tilde{s}_{3'}$}}%
\put(13.0000,-11.5000){\makebox(0,0){$\tilde{s}_{2'}$}}%
\put(3.0000,-31.5000){\makebox(0,0){$\tilde{s}_{2'}$}}%
\put(15.0000,-18.0000){\makebox(0,0){$\tilde{s}_{1'}$}}%
\put(9.0000,-23.5000){\makebox(0,0){$\tilde{s}_{1'}$}}%
\put(5.0000,-27.5000){\makebox(0,0){$\tilde{s}_{1'}$}}%
\put(3.0000,-7.5000){\makebox(0,0){$\tilde{s}_{1'}$}}%
\put(7.0000,-11.5000){\makebox(0,0){$\tilde{s}_{1'}$}}%
\end{picture}}%
\qquad\qquad
\input{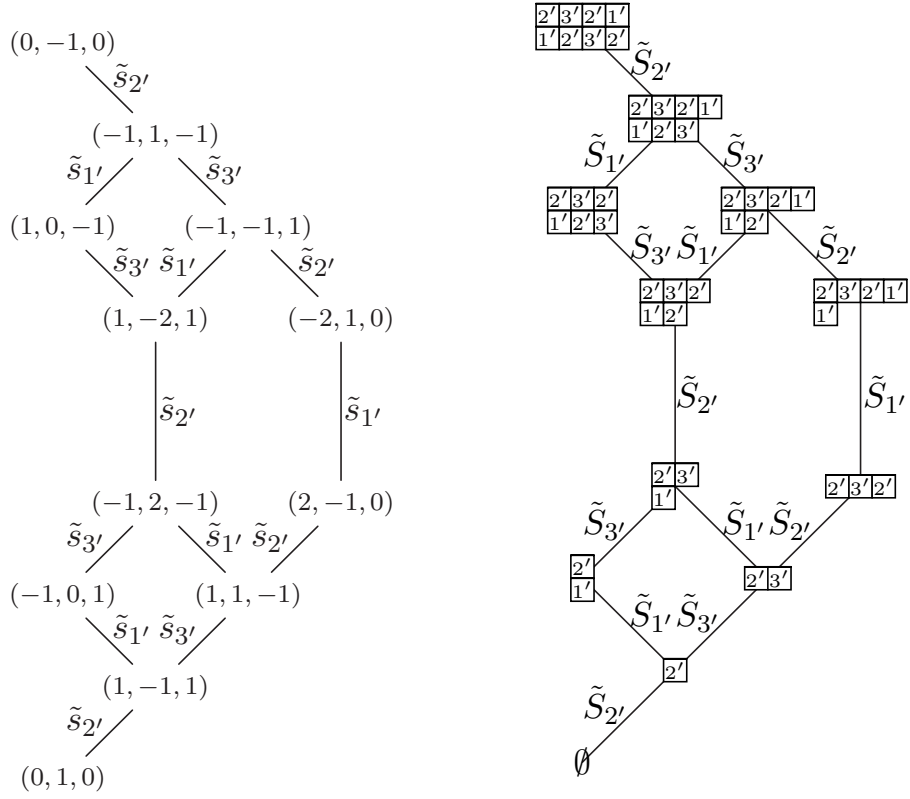}
\vspace{10pt}
\caption{$(\tilde{W}\tilde{\Lambda}_{2'},\le_w)$ and $(\tilde{\mathcal{F}}(P_{\Lambda_{2}}),\tilde{\unlhd})$ of type $C_3$}\label{fig:fold color filter,orbit}.
\end{figure}


\section{Proof of Theorem \ref{iso fold filter,orbit}.}\label{Sec:main-proof}

\ytableausetup{mathmode,boxsize=1.8em}

Keep the notation and setting in the previous section.

\begin{defi}\label{S fold}
For $p\in J$, we define $\hat{S}_p:\mathcal{F}(P_{\lambda})\rightarrow \mathcal{F}(P_{\lambda})$ as follows:
\begin{enumerate}\renewcommand{\labelenumi}{(\arabic{enumi})}
\item If $p$ satisfies the orthogonality condition, then
$$\hat{S}_p:=\prod_{k\in p}S_k;$$
we see by Lemma \ref{action filter,orbit} that $\hat{S}_p$ does not depend on the order of the product of $S_k$'s.
\item If $p$ does not satisfy the orthogonality condition, that is, if $\Lieg$ is of type $A_{2n}$ and $p=\{n,n+1\}$ (see page \pageref{not orth}), then
$$\hat{S}_p:= S_n S_{n+1} S_n = S_{n+1} S_n S_{n+1} ;$$
the second equality follows from Lemma \ref{action filter,orbit}, together with $s_n s_{n+1} s_n = s_{n+1} s_n s_{n+1}$.
\end{enumerate}
\end{defi}

We need the following fact to prove Lemma \ref{property S fold} below.

\begin{prop}[{\cite[page 23]{Dave}}]\label{property filter}
Let $P$ be an arbitrary poset, and let $F\in \mathcal{F}(P)$ be an order filter of $P$.
\begin{enumerate}\renewcommand{\labelenumi}{(\arabic{enumi})}
\item[(1)]\ For $x\in F$, $x$ is a minimal element of $F$ if and only if $F\setminus \{x\}$ is an order filter.
\item[(2)]\ For $x\in F$, $x$ is a maximal element of $P\setminus F$ if and only if $F\cup \{x\}$ is an order filter.
\end{enumerate}
\end{prop}

\begin{lemm}\label{property S fold}
Let $\mu\in W\lambda$, and set $F:=f(\mu)\in \mathcal{F}(P_\lambda)$.
It holds that
\begin{eqnarray}\label{eq:p-folds}
\tilde{S}_p(F)=\hat{S}_p (F) \text{ for all } p\in J.
\end{eqnarray}
\end{lemm}

\begin{proof}
First, we assume that $p\in J$ satisfies the orthogonality condition.
The case that $\# p=1$ is easy.
Assume that $\# p=2$ (the proof for the case that $\# p=3$ is similar).
Let we write $p$ as: $p=\{i,j\}$, with $i,j\in I, i\not= j$.
We deduce by Lemma \ref{property f-invo,c-d-comp} that for each $k\in p=\{i,j\}$, $S_k(F)$ satisfies one of the following:
\begin{enumerate}
\renewcommand{\labelenumi}{(\roman{enumi})}
\item $S_k(F)=A_k(F)=F\sqcup \{x_k\}$ for some $x_k\in P_{\lambda}\setminus F$; in this case, $R_k(F)=F$.
\item $S_k(F)=R_k(F)=F\setminus \{x_k\}$ for some $x_k\in F$; in this case, $A_k(F)=F$.
\item $S_k(F)=A_k(F)=R_k(F)=F$.
\end{enumerate}

\noindent Here, we give a proof only for the case that both $S_i(F)$ and $S_j(F)$ satisfy (i); the proofs for the other cases are similar.
In this case, there exist $x_i,x_j\in P_{\lambda}\setminus F$ such that $S_i(F)=F\sqcup \{x_i\}$ and $S_j(F)=F\sqcup \{x_j\}$; note that $\kappa(x_i)=i$ and $\kappa(x_j)=j$.
By Definition \ref{S fold}, we have $S_j(F\sqcup \{x_i\})=S_jS_i(F)=S_iS_j(F)=S_i(F\sqcup \{x_j\})$.
Since $\kappa(x_i)=i$ and $i\not=j$, we see from the definition of $S_j$ that when we apply $S_j$ to $F\sqcup \{x_i\}$, $x_i$ is not removed.
Hence, $x_i\in S_j(F\sqcup \{x_i\})=S_jS_i(F)$.
Similarly, $x_j\in S_jS_i(F)$.
Since the symmetric difference of $S_jS_i(F)$ and $F$ has at most two element by Lemma \ref{property f-invo,c-d-comp}, we see that $S_jS_i(F)=F\sqcup \{x_i\}\sqcup \{x_j\}$.
Therefore, it suffices to show that $\tilde{S}_p(F)=F\sqcup \{x_i\}\sqcup \{x_j\}$.

Suppose, for a contradiction, that $F\supsetneq \tilde{R}_p(F)$.
Let $y$ be a minimal element of $F\setminus \tilde{R}_p(F)$.
Because $\tilde{R}_p(F)$ is an order filter by the definition of $\tilde{R}_p$, we deduce that $y$ is a minimal element of $F$.
Hence, by Proposition \ref{property filter} (1), $F\setminus \{y\}$ is an order filter of $P_{\lambda}$.
Note that $\tilde{\kappa}(y)=p$, and recall that $\tilde{\kappa}(y)=p$ if and only if $\kappa(y)=i$ or $\kappa(y)=j$.
Assume that $\kappa(y)=i$.
Since $F\setminus\{y\}$ is an order filter of $P_\lambda$ satisfying $F\setminus(F\setminus\{y\})=\{y\}\subseteq\kappa^{-1}(\{i\})$, we see by the definition of $R_i$ that $R_i(F)\not=F$.
Similarly, if $\kappa(y)=j$, then $R_j(F)\not=F$.
Thus we conclude that $R_i(F)\ne F$ or $R_j(F)\ne F$.
However, this contradicts the assumption that both $S_i(F)$ and $S_j(F)$ satisfy (i).
Therefore, $\tilde{R}_p(F)=F$, and hence $\tilde{S}_p(F)=\tilde{A}_p(F)$.
Since $F\sqcup \{x_i\}\sqcup \{x_j\} =S_jS_i(F)$ is an order filter of $P_\lambda$, we see by the definition of $\tilde{A}_p(F)$ that $F\sqcup \{x_i\}\sqcup \{x_j\}\subseteq \tilde{A}_p(F)=\tilde{S}_p(F)$.

Suppose, for a contradiction, that $\tilde{S}_p(F)\supsetneq F\sqcup \{x_i\}\sqcup \{x_j\}$.
Since $F\sqcup \{x_j\},F\sqcup \{x_i\}\in \mathcal{F}(P_\lambda)$, it follows from Proposition \ref{property filter} (1) that $x_i$ and $x_j$ are minimal elements of $F\sqcup \{x_i\}\sqcup \{x_j\}$.
Let $z$ be a maximal element of $\tilde{S}_p(F)\setminus (F\sqcup \{x_i\}\sqcup \{x_j\})$; note that $\tilde{\kappa}(z)=p$, which implies that $\kappa(z)\in p=\{i,j\}$.
If $z$ and $x_i$ are comparable, then $z\rightarrow x_i$ because $F\sqcup \{x_i\}\sqcup \{x_j\}$ is an order filter, and $x_i$ is a minimal element of $F\sqcup \{x_i\}\sqcup \{x_j\}$ as seen above.
By Proposition \ref{property I-c-d-comp}, $\kappa(z)\in\{i,j\}$ and $\kappa(x_i)=i$ are adjacent in the Dynkin diagram of $\Lieg$.
However, this contradicts that $p$ satisfies the orthogonality condition.
Thus, $z$ and $x_i$ are incomparable.
Similarly, we can show that $z$ and $x_j$ are incomparable.
Thus, $z$ is a maximal element of $\tilde{S}_p(F)\setminus F$.
Since $\tilde{S}_p(F)$ is an order filter of $P_\lambda$, we see that $z$ is a maximal element of $P_{\lambda}\setminus F$.
Hence, by Proposition \ref{property filter} (2), $F\sqcup \{ z\}$ is an order filter of $P_{\lambda}$.
Since $\kappa(z)\in p =\{i,j\}$, we see by the definitions of $A_i$ and $A_j$ that $z$ is contained in either $A_i(F)$ or $A_j(F)$.
However, this contradicts the assumption that both $S_i(F)$ and $S_j(F)$ satisfy (i).
Therefore, we obtain $F\sqcup \{x_i\}\sqcup \{x_j\}=\tilde{S}_p(F)$, as desired.

Next, we assume that $p$ does not satisfy the orthogonality condition, that is, $\Lieg$ is of type $A_{2n}$ and $p=\{n,n+1\}$.
Let $\lambda=\Lambda_i$.
In this case, the corresponding d-complete poset $P_\lambda$ is $Y_{i,2n-i+1}$ (see Example \ref{diagram} (1)), and its $I$-coloring $\kappa:P_\lambda\rightarrow I$ is given as follows (see also Figure \ref{fig:c-d-comp}):
\vspace{10pt}
\ytableausetup{mathmode,boxsize=2.4em}
\begin{center}
$(P_{\lambda},\le,\kappa,I)=$
\begin{ytableau}
i&\none[\dots]&\ast&\ast&\ast&\ast&\none[\dots]&2n\\
\none[\vdots]&\none&\none[\vdots]&\none[\vdots]&\none[\vdots]&\none[\vdots]&\none&\none[\vdots]\\
\ast&\none[\dots]&n&\scriptstyle n+1&\scriptstyle n+2&\scriptstyle n+3&\none[\dots]&\ast\\
\ast&\none[\dots]&\scriptstyle n-1&n&\scriptstyle n+1&\scriptstyle n+2&\none[\dots]&\ast\\
\ast&\none[\dots]&\scriptstyle n-2&\scriptstyle n-1&n&\scriptstyle n+1&\none[\dots]&\ast\\
\none[\vdots]&\none&\none[\vdots]&\none[\vdots]&\none[\vdots]&\none[\vdots]&\none&\none[\vdots]\\
1&\none[\dots]&\ast&\ast&\ast&\ast&\none[\dots]&\scriptscriptstyle 2n-i+1\\
\end{ytableau}\ .
\end{center}
\vspace{10pt}

\noindent In this proof, the cells having the color $n$ or $n+1$ are important; if $1\le i\le n$ (resp., $n+1\le i\le 2n$), then $\kappa^{-1}(\{n\})=\{(1,n-i+1),(2,n-i+2),\ldots,(i,n)\}$ and $\kappa^{-1}(\{n+1\})=\{(1,n-i+2),(2,n-i+3),\ldots,(i,n+1)\}$ (resp., $\kappa^{-1}(\{n\})=\{(i-n+1,1),(i-n+2,2),\ldots,(n+1,2n-i+1)\}$ and $\kappa^{-1}(\{n+1\})=\{(i-n,1),(i-n+1,2),\ldots,(n,2n-i+1)\}$).
Notice that the subset $\kappa^{-1}(\{n,n+1\})\subset P_\lambda$ is a totally order set.
Similarly to the case that $p$ satisfies the orthogonality condition, each of $S_n(F)$ and $S_{n+1}(F)$ satisfies one of (i),(ii),(iii).
Suppose, for a contradiction, that both $S_n(F)$ and $S_{n+1}(F)$ satisfy (i).
Then, there exist $x_n,x_{n+1}$ such that both $x_n$ and $x_{n+1}$ are maximal elements of $P_\lambda\setminus F$, and $\kappa(x_n)=n,\kappa(x_{n+1})=n+1$.
However, this contradict the fact that $\kappa^{-1}(\{n,n+1\})$ is a totally order set.
Therefore, the case that both $S_n(F)$ and $S_{n+1}(F)$ satisfy (i) does not happen.
Similarly, we deduce that the case that both $S_n(F)$ and $S_{n+1}(F)$ satisfy (ii) does not happen.
So, it suffices to consider the other 7 cases.

Now, we give a proof only for the case that $S_n(F)$ satisfies (i), and $S_{n+1}(F)$ satisfies (iii); the proofs for the other cases are similar.
Then, under the description mentioned at the end of Section 2, $F$ has a ``block'' of the following form:
\begin{center}
$F=$
\ytableausetup{mathmode,boxsize=1.8em}
\begin{ytableau}
\none&\none[\vdots]&\none[\vdots]&\none[\vdots]&\none[\vdots]&\none\\
\none[\dots]&n&\scriptstyle n+1&\scriptstyle n+2&*(lightgray)\scriptstyle n+3&\none[\dots]\\
\none[\dots]&\scriptstyle n-1&*(gray)n&*(gray)\scriptstyle n+1&*(gray)\scriptstyle n+2&\none[\dots]\\
\none[\dots]&*(lightgray)\scriptstyle n-2&*(gray)\scriptstyle n-1&*(gray)n&*(gray)\scriptstyle n+1&\none[\dots]\\
\none&\none[\vdots]&\none[\vdots]&\none[\vdots]&\none[\vdots]&\none
\end{ytableau}
\end{center}

\noindent Here, each element corresponding to the right-gray cell (with the color $n+3$ or $n-2$) is not necessarily an element of $F$.
Then, $\hat{S}_p(F)$ and $\tilde{S}_p(F)$ are as follows:

\begin{eqnarray*}
\hat{S}_p\left(\vbox to 25pt{} \begin{ytableau}
\none&\none[\vdots]&\none[\vdots]&\none[\vdots]&\none[\vdots]&\none\\
\none[\dots]&n&\scriptstyle n+1&\scriptstyle n+2&*(lightgray)\scriptstyle n+3&\none[\dots]\\
\none[\dots]&\scriptstyle n-1&*(gray)n&*(gray)\scriptstyle n+1&*(gray)\scriptstyle n+2&\none[\dots]\\
\none[\dots]&*(lightgray)\scriptstyle n-2&*(gray)\scriptstyle n-1&*(gray)n&*(gray)\scriptstyle n+1&\none[\dots]\\
\none&\none[\vdots]&\none[\vdots]&\none[\vdots]&\none[\vdots]&\none
\end{ytableau}\right)
&=&S_nS_{n+1}S_n\left(\vbox to 25pt{} \begin{ytableau}
\none&\none[\vdots]&\none[\vdots]&\none[\vdots]&\none[\vdots]&\none\\
\none[\dots]&n&\scriptstyle n+1&\scriptstyle n+2&*(lightgray)\scriptstyle n+3&\none[\dots]\\
\none[\dots]&\scriptstyle n-1&*(gray)n&*(gray)\scriptstyle n+1&*(gray)\scriptstyle n+2&\none[\dots]\\
\none[\dots]&*(lightgray)\scriptstyle n-2&*(gray)\scriptstyle n-1&*(gray)n&*(gray)\scriptstyle n+1&\none[\dots]\\
\none&\none[\vdots]&\none[\vdots]&\none[\vdots]&\none[\vdots]&\none
\end{ytableau}\right)\\
&=&S_nS_{n+1}
\left(\vbox to 25pt{} \begin{ytableau}
\none&\none[\vdots]&\none[\vdots]&\none[\vdots]&\none[\vdots]&\none\\
\none[\dots]&n&\scriptstyle n+1&\scriptstyle n+2&*(lightgray)\scriptstyle n+3&\none[\dots]\\
\none[\dots]&\scriptstyle n-1&n&*(gray)\scriptstyle n+1&*(gray)\scriptstyle n+2&\none[\dots]\\
\none[\dots]&*(lightgray)\scriptstyle n-2&*(gray)\scriptstyle n-1&*(gray)n&*(gray)\scriptstyle n+1&\none[\dots]\\
\none&\none[\vdots]&\none[\vdots]&\none[\vdots]&\none[\vdots]&\none
\end{ytableau}\right)\\
&=&S_n
\left(\vbox to 25pt{} \begin{ytableau}
\none&\none[\vdots]&\none[\vdots]&\none[\vdots]&\none[\vdots]&\none\\
\none[\dots]&n&\scriptstyle n+1&\scriptstyle n+2&*(lightgray)\scriptstyle n+3&\none[\dots]\\
\none[\dots]&\scriptstyle n-1&n&\scriptstyle n+1&*(gray)\scriptstyle n+2&\none[\dots]\\
\none[\dots]&*(lightgray)\scriptstyle n-2&*(gray)\scriptstyle n-1&*(gray)n&*(gray)\scriptstyle n+1&\none[\dots]\\
\none&\none[\vdots]&\none[\vdots]&\none[\vdots]&\none[\vdots]&\none
\end{ytableau}\right)\\
&=&
\begin{ytableau}
\none&\none[\vdots]&\none[\vdots]&\none[\vdots]&\none[\vdots]&\none\\
\none[\dots]&n&\scriptstyle n+1&\scriptstyle n+2&*(lightgray)\scriptstyle n+3&\none[\dots]\\
\none[\dots]&\scriptstyle n-1&n&\scriptstyle n+1&*(gray)\scriptstyle n+2&\none[\dots]\\
\none[\dots]&*(lightgray)\scriptstyle n-2&*(gray)\scriptstyle n-1&*(gray)n&*(gray)\scriptstyle n+1&\none[\dots]\\
\none&\none[\vdots]&\none[\vdots]&\none[\vdots]&\none[\vdots]&\none
\end{ytableau}\ .
\end{eqnarray*}

\begin{eqnarray*}
\tilde{S}_p\left(\vbox to 25pt{} \begin{ytableau}
\none&\none[\vdots]&\none[\vdots]&\none[\vdots]&\none[\vdots]&\none\\
\none[\dots]&n&\scriptstyle n+1&\scriptstyle n+2&*(lightgray)\scriptstyle n+3&\none[\dots]\\
\none[\dots]&\scriptstyle n-1&*(gray)n&*(gray)\scriptstyle n+1&*(gray)\scriptstyle n+2&\none[\dots]\\
\none[\dots]&*(lightgray)\scriptstyle n-2&*(gray)\scriptstyle n-1&*(gray)n&*(gray)\scriptstyle n+1&\none[\dots]\\
\none&\none[\vdots]&\none[\vdots]&\none[\vdots]&\none[\vdots]&\none
\end{ytableau}\right)=
\begin{ytableau}
\none&\none[\vdots]&\none[\vdots]&\none[\vdots]&\none[\vdots]&\none\\
\none[\dots]&n&\scriptstyle n+1&\scriptstyle n+2&*(lightgray)\scriptstyle n+3&\none[\dots]\\
\none[\dots]&\scriptstyle n-1&n&\scriptstyle n+1&*(gray)\scriptstyle n+2&\none[\dots]\\
\none[\dots]&*(lightgray)\scriptstyle n-2&*(gray)\scriptstyle n-1&*(gray)n&*(gray)\scriptstyle n+1&\none[\dots]\\
\none&\none[\vdots]&\none[\vdots]&\none[\vdots]&\none[\vdots]&\none
\end{ytableau}
\end{eqnarray*}

\noindent Thus we obtain $\tilde{S}_p(F)=\hat{S}_p(F)$, as desired.
\end{proof}

\begin{lemm}\label{action fold filter,orbit}
For $\mu\in \hat{W}\lambda$ and $p\in J$,
$$\tilde{f}(\tilde{s}_p(\text{res}(\mu)))=\tilde{S}_p(\tilde{f}(\text{res}(\mu))).$$
In particular,
\begin{eqnarray}\label{tilde-F}
\tilde{\mathcal{F}}(P_{\lambda})=\{\tilde{S}_{p_n}\cdots\tilde{S}_{p_2}\tilde{S}_{p_1}(f(\lambda))\,|\,n\ge 0,p_k\in J(1\le k\le n)\}.
\end{eqnarray}
\end{lemm}

\begin{proof}
We compute that
\begin{eqnarray*}
\tilde{f}(\tilde{s}_p(\text{res}(\mu)))&=&\tilde{f}(\text{res}(\hat{s}_p (\mu)))\,\qquad(\rm by\ Lemma\ \ref{property weyl fold})\\
&=&f(\hat{s}_p (\mu))\,\qquad\qquad(\rm by\ the\ definition\ of\ \it\tilde{f})\\
&=&\hat{S}_p(f(\mu))\qquad\qquad(\rm by\ Corollary\ \ref{action filter,orbit})\\
&=&\hat{S}_p(\tilde{f}(\text{res}(\mu)))\qquad(\rm by\ the\ definition\ of\ \it \tilde{f})\\
&=&\tilde{S}_p(\tilde{f}(\text{res}(\mu)))\qquad(\rm by\ Lemma\ \ref{property S fold}).
\end{eqnarray*}
\end{proof}

\begin{proof}[Proof of Theorem \ref{iso fold filter,orbit}.]
(1) By the definitions of $\le_w$ and $\tilde{\unlhd}$, it suffices to show that for $\mu\in \hat{W}\lambda$ and $\ p\in J$, $\text{res}(\mu)<_w \tilde{s}_p(\text{res}(\mu))$ if and only if $\tilde{f}(\text{res}(\mu))\ \tilde{\lhd}\ \tilde{S}_p(\tilde{f}(\text{res}(\mu)))$.
First, we assume that $\text{res}(\mu)<_w \tilde{s}_p(\text{res}(\mu))=\text{res}(\hat{s}_p(\mu))$.
Because $\text{res}(\mu)(H_p)>0$, there exists $i \in p$ such that $\mu(h_i)>0$.
Then we deduce by Lemma \ref{sign coincide}(2) that $\mu <_w \hat{s}_p(\mu)$ in $(W\lambda,\le_w)$.
By the definition of $<_w$ and $<_s$, we have $\mu <_s \hat{s}_p(\mu)$ in $(W\lambda,\le_s)$.
So we compute
\begin{eqnarray*}
f(\mu)&\subsetneq&f(\hat{s}_p(\mu))\,\qquad\qquad(\rm by\ Proposition\ \ref{iso filter,orbit})\\
&=&\hat{S}_p(f(\mu))\qquad\qquad(\rm by\ Corollary\ \ref{action filter,orbit})\\
&=&\tilde{S}_p(\tilde{f}(\text{res}(\mu)))\qquad(\rm by\ the\ definition\ of\ \it \tilde{f}\  \rm and\  Lemma\  \ref{property S fold}).
\end{eqnarray*}
Therefore, we obtain $\tilde{f}(\text{res}(\mu))\ \tilde{\lhd}\ \tilde{S}_p(\tilde{f}(\text{res}(\mu)))$, as desire.

Next, we assume that $\tilde{f}(\text{res}(\mu))\ \tilde{\lhd}\ \tilde{S}_p(\tilde{f}(\text{res}(\mu)))$.
Then we have $\tilde{f}(\text{res}(\mu))\subsetneq\tilde{S}_p(\tilde{f}(\text{res}(\mu)))$.
Since $\tilde{f}(\text{res}(\mu))=f(\mu)$ and $\tilde{S}_p(\tilde{f}(\text{res}(\mu)))=f(\hat{s}_p(\mu))$ as seen above, we get $f(\mu)\subsetneq f(\hat{s}_p(\mu))$.
Hence, by Proposition \ref{iso filter,orbit}, $\mu<_s \hat{s}_p(\mu)$ in $(W\lambda,\le_s)$.
Write $\hat{s}_p(\mu)$ as: $\hat{s}_p(\mu)=\mu-\sum_{i\in p}m_i\alpha_i$; since $\mu<_s \hat{s}_p(\mu)$, we see that $m_i\ge 0$ for all $i\in I$, and $m:=\sum_{i\in p}m_i>0$.
Because $\tilde{s}_p(\text{res}(\mu))=\text{res}(\mu)-m\beta_p$, we obtain $\text{res}(\mu)<_w\tilde{s}_p(\text{res}(\mu))$, as desire.

(2) 
For $\mu_1,\mu_2\in \hat{W}\lambda$, We deduce
\begin{eqnarray*}
&&\text{res}(\mu_1)<_s \text{res}(\mu_2)\\
&\Leftrightarrow&\mu_1<_s \mu_2\qquad\qquad\qquad\qquad\ (\rm by\ Lemma\ \ref{bruhat fold})\\
&\Leftrightarrow&f(\mu_1)\subset f(\mu_2)\qquad\qquad\qquad(\rm by\ Proposition\ \ref{iso filter,orbit})\\
&\Leftrightarrow&\tilde{f}(\text{res}(\mu_1))\subset \tilde{f}(\text{res}(\mu_2))\qquad(\rm by\ the\ definition\ of\ \it \tilde{f}).
\end{eqnarray*}
\end{proof}

\section{Explicit description of $\tilde{\mathcal{F}}(P_{\lambda})$.}\label{Sec:tildeF}

Keep the notation and setting in Section \ref{Sec:J-color}.
We give an explicit description of $\tilde{\mathcal{F}}(P_{\lambda})$ in the case that $\Lieg$ is of type $A_n$; in fact, our description, Theorem \ref{Fexplicid} below, and its proof are essentially restatements of \cite[Theorem 1.1 and its proof\,]{Mo}; however, we give a proof (in terms of our notation) for the convenience of the readers.

Assume that $\Lieg$ is of type $A_n$, and $\lambda=\Lambda_m$ with $1 \le m \le (n+1)/2$.
We regard $P_{\lambda}$ as a rectangular Young diagram $Y_{m,n-m+1}$ (see Example \ref{diagram}).
Note that $\kappa((i,j)) = j-i+m$ and $\tilde{\kappa}((i,j)) = (\min\{j-i+m, i-j+n-m+1\})'$.

For $i,j,p\in\mathbb{Z}$, we set $[i,j] := \{k \in \mathbb{Z}\mid i \le k \le j\}$ and $\binom{[i,j]}{p} := \{I\subseteq [i,j] \mid  \#I = p\}$.

\begin{defi}\label{partition}
Let $Y \in \mathcal{F}(Y_{m,n-m+1})$.
For $1\le l\le m$, we set $\boldsymbol{k}_l := \#\{j \mid (l,j) \in Y\}$, and $\boldsymbol{k}(Y) := (\boldsymbol{k}_1,\ldots,\boldsymbol{k}_m)$.
Also, we set $\mathcal{I}(Y) := \{\boldsymbol{k}_i + m+1-i\mid i \in [1,m]\}, \overline{\mathcal{I}(Y)} := \{n+2-i\mid i \in \mathcal{I}(Y) \}$.
Observe that the map $\mathcal{I}: \mathcal{F}(Y_{m,n-m+1}) \to \binom{[1,n+1]}{m}, Y \mapsto \mathcal{I}(Y),$ is a bijection.
\end{defi}

\begin{lemm}\label{IandS}
Let $Y \in \mathcal{F}(Y_{m,n-m+1})$, and $k \in [1,n]$.
Then,
\begin{enumerate}\renewcommand{\labelenumi}{(\arabic{enumi})}
\item[(1)] $k \in \mathcal{I}(Y)$ and $k+1 \not\in \mathcal{I}(Y)$ if and only if $S_k(Y) \supset Y$;
\item[(2)] $k \not\in \mathcal{I}(Y)$ and $k+1 \in \mathcal{I}(Y)$ if and only if $S_k(Y) \subset Y$;
\item[(3)] $k,k+1 \in \mathcal{I}(Y)$ or $k,k+1 \not\in \mathcal{I}(Y)$ if and only if $S_k(Y) = Y$.
\end{enumerate}
\end{lemm}
\begin{proof}
Notice that for $Y \in \mathcal{F}(Y_{m,n-m+1})$ and $k \in [1,n]$, $S_k(Y)$ satisfies one of the following (see Lemma \ref{property f-invo,c-d-comp} and Corollary \ref{action filter,orbit}):
\begin{enumerate}
\renewcommand{\labelenumi}{(\roman{enumi})}
\item $S_k(Y)=A_k(Y)=Y\sqcup \{(i,j)\}$ for some $(i,j)\in Y_{m,n-m+1}\setminus Y$.
\item $S_k(Y)=R_k(Y)=Y\setminus \{(i,j)\}$ for some $(i,j)\in Y$.
\item $S_k(Y)=A_k(Y)=R_k(Y)=Y$.
\end{enumerate}

(1) First, we show the ``only if'' part.
Because $k \in \mathcal{I}(Y)$, there exists $i \in [1,m]$ such that $k = \boldsymbol{k}_i + m+1-i$.
Then we see that $(i,\boldsymbol{k}_i) = (i,k-m-1+i) \in Y$ or $\boldsymbol{k}_i = 0$. 
In both cases, we get $(i,k-m+i) \not\in Y$.
Note that $\kappa(i,k-m+i) = k$.
If $i = 1$, then we get $Y\sqcup\{(i,k-m+i)\}\in \mathcal{F}(Y_{m,n-m+1})$ and $S_k(Y) = Y\sqcup\{(i,k-m+i)\} \supset Y$.
If $i > 1$, then $\boldsymbol{k}_{i-1} + m+1-(i-1) > k+1$ by $k+1 \not\in \mathcal{I}(Y)$.
Thus we have $k-m+i \le \boldsymbol{k}_{i-1}$ and $(i-1,k-m+i) \in Y$.
Hence we get $Y\sqcup\{(i,k-m+i)\}\in \mathcal{F}(Y_{m,n-m+1})$ and $S_k(Y) = Y\sqcup\{(i,k-m+i)\} \supset Y$.

Next, we show the ``if'' part.
Because $S_k(Y) \supset Y$, there exists $(i,j) \in S_k(Y)$ such that $S_k(Y) = Y \sqcup\{(i,j)\}$ and $\kappa(i,j) = j-i+m = k$.
Then we see that $(i,j-1)\in Y$ or $j-1 = 0$.
In both cases, we get $\boldsymbol{k}_i = j-1 = i-m+k-1$.
Hence, $k = \boldsymbol{k}_i + m+1-i \in \mathcal{I}(Y)$.
If $i = 1$, then $\max(\mathcal{I}(Y)) = k$, and hence $k+1 \not\in \mathcal{I}(Y)$.
If $i > 1$, then $(i-1,j) \in Y$ and $\boldsymbol{k}_{i-1} \ge j$.
Thus we obtain $\boldsymbol{k}_{i-1} + m+1-(i-1) \ge j + m+1-i+1=k+2$, which implies that $k+1 \not\in \mathcal{I}(Y)$.

(2) Similar to part (1).

(3) Since $S_k(Y)$ satisfies one of (i)-(iii), the assertion is obvious from parts (1) and (2).
\end{proof}

\begin{rema}\label{IandS2}
By Lemma \ref{IandS}, if $k \in \mathcal{I}(Y)$ and $k+1 \not\in \mathcal{I}(Y)$, then $k \not\in \mathcal{I}(S_k(Y))$ and $k+1 \in \mathcal{I}(S_k(Y))$.
Moreover, either $k' \in \mathcal{I}(Y), k' \in \mathcal{I}(S_k(Y))$ or $k' \not\in \mathcal{I}(Y), k' \not\in \mathcal{I}(S_k(Y))$ for $k' \in [1,n+1]$ with $k' \not = k,k+1$.

\end{rema}

For $n \in \mathbb{Z}_{>0}$ and $m \in \mathbb{Z}_{>0}$ such that $1 \le m \le (n+1)/2$, we set $\mathcal{SS}(Y_{m,n-m+1}) := \{Y \in \mathcal{F}(Y_{m,n-m+1}) \mid \mathcal{I}(Y) \cap \overline{\mathcal{I}(Y)} = \emptyset \}$.

\begin{theo}[{cf. \cite[Theorem 1.1]{Mo}}]\label{Fexplicid}
It holds that $\tilde{\mathcal{F}}(Y_{m,n-m+1}) = \mathcal{SS}(Y_{m,n-m+1}).$
\end{theo}
\begin{proof}
We will show that $Y \in \tilde{\mathcal{F}}(Y_{m,n-m+1})$ if and only if $Y \in \mathcal{SS}(Y_{m,n-m+1})$ by induction on $\#Y$.
If $\#Y = 0$, then $Y = \emptyset$.
It is obvious that $\emptyset \in \tilde{\mathcal{F}}(Y_{m,n-m+1})$.
Also, because $I(\emptyset) = \{1,2,\ldots,m\}$ and $I(\emptyset) = \{n+1,n,\ldots,n+2-m\}$, with $m < n+2-m$, it follows that $\mathcal{I}(\emptyset) \cap \overline{\mathcal{I}(\emptyset)} = \emptyset$, and hence $\emptyset \in \mathcal{SS}(Y_{m,n-m+1})$.

Assume that $\#Y > 0$.
First, we will show  the ``only if'' part.
Because $Y \not= \emptyset$, there exists $p\in J$ such that $\tilde{S}_p(Y) \subset Y$.
Since $Y \in \tilde{\mathcal{F}}(Y_{m,n-m+1})$, we have $\tilde{S}_p(Y) \in \tilde{\mathcal{F}}(Y_{m,n-m+1})$.
By the induction hypothesis, it follows that $\tilde{S}_p(Y) \in \mathcal{SS}(Y_{m,n-m+1})$.
Here, we give a proof only for the case that $\#p = 2$; the proof for the case that $\#p = 1$ is similar (and simpler).
Assume that $p$ satisfies the orthogonality condition.
We write $p$ as: $p = \{i,n+1-i\}$ with $i\not=n+1-i$.
By Lemma \ref{sign coincide}, $\tilde{S}_p(Y)$ satisfies one of the following:
\begin{enumerate}
\renewcommand{\labelenumi}{(\roman{enumi})}
\item $\tilde{S}_p(Y) \subset S_i\tilde{S}_p(Y), \tilde{S}_p(Y) = S_{n+1-i}\tilde{S}_p(Y)$.
\item $\tilde{S}_p(Y) = S_i\tilde{S}_p(Y), \tilde{S}_p(Y) \subset S_{n+1-i}\tilde{S}_p(Y)$.
\item $\tilde{S}_p(Y) \subset S_i\tilde{S}_p(Y), \tilde{S}_p(Y) \subset S_{n+1-i}\tilde{S}_p(Y)$.
\end{enumerate}
We see by Lemma \ref{IandS} that (i) (resp., (ii), (iii)) holds if and only if the following (i)' (resp., (ii)', (iii)') holds:
\begin{enumerate}
\renewcommand{\labelenumi}{(\roman{enumi})'}
\item $i\in\mathcal{I}(\tilde{S}_p(Y)), i+1\not\in\mathcal{I}(\tilde{S}_p(Y)), n+1-i\not\in\mathcal{I}(\tilde{S}_p(Y)), n+2-i\not\in\mathcal{I}(\tilde{S}_p(Y))$.
\item $i\not\in\mathcal{I}(\tilde{S}_p(Y)), i+1\not\in\mathcal{I}(\tilde{S}_p(Y)), n+1-i\in\mathcal{I}(\tilde{S}_p(Y)), n+2-i\not\in\mathcal{I}(\tilde{S}_p(Y))$.
\item $i\in\mathcal{I}(\tilde{S}_p(Y)), i+1\not\in\mathcal{I}(\tilde{S}_p(Y)), n+1-i\in\mathcal{I}(\tilde{S}_p(Y)), n+2-i\not\in\mathcal{I}(\tilde{S}_p(Y))$.
\end{enumerate}
Moreover, it can be easily checked that (i)' (resp., (ii)', (iii)') holds if and only if the following (i)'' (resp., (ii)'', (iii)'') holds:
\begin{enumerate}
\renewcommand{\labelenumi}{(\roman{enumi})''}
\item $i\not\in\mathcal{I}(Y), i+1\in\mathcal{I}(Y), n+1-i\not\in\mathcal{I}(Y), n+2-i\not\in\mathcal{I}(Y)$,
\item $i\not\in\mathcal{I}(Y), i+1\not\in\mathcal{I}(Y), n+1-i\not\in\mathcal{I}(Y), n+2-i\in\mathcal{I}(Y)$,
\item $i\not\in\mathcal{I}(Y), i+1\in\mathcal{I}(Y), n+1-i\not\in\mathcal{I}(Y), n+2-i\in\mathcal{I}(Y)$.
\end{enumerate}
By Remark \ref{IandS2}, we obtain $Y \in \mathcal{SS}(Y_{m,n-m+1})$ for any cases.
Assume that $p$ does not satisfy the orthogonality condition; in this case, $n$ is even, and $p = \{i,i+1\}$ with $i = n/2$.
By Lemmas \ref{sign coincide} and \ref{IandS}, $\tilde{S}_p(Y)$ satisfies $i\in\mathcal{I}(\tilde{S}_p(Y)), i+1\not\in\mathcal{I}(\tilde{S}_p(Y))$, and $ i+2\not\in\mathcal{I}(\tilde{S}_p(Y))$.
Also, $Y$ satisfies $i\not\in\mathcal{I}(Y), i+1\not\in\mathcal{I}(Y)$, and $i+2\in\mathcal{I}(Y)$.
Thus we obtain $Y \in \mathcal{SS}(Y_{m,n-m+1})$, as desired.

Next, we will show  the ``if'' part.
Because $Y \not= \emptyset$, there exists $k \in [1,n]$ such that $k \not\in \mathcal{I}(Y)$ and $k+1 \in \mathcal{I}(Y)$; we set $p := \{k, n+1-k\} \in J$.
Let $Y' \in \mathcal{F}(Y_{m,n-m+1})$ be such that $\mathcal{I}(Y') = \mathcal{I}(Y)\sqcup\{k\}\setminus\{k+1\}$; note that $\#Y' = \#Y -1$.
Assume that $n+2-k \not\in \mathcal{I}(Y')$.
By Remark \ref{IandS2}, we have $Y' \in \mathcal{SS}(Y_{m,n-m+1})$.
By the induction hypothesis, it follows that $Y' \in \tilde{\mathcal{F}}(Y_{m,n-m+1})$.
Notice that $n/2+1 \not\in \mathcal{I}(Y)$, because $n+2-(n/2+1) = n/2+1$.
Because $k+1\in \mathcal{I}(Y)$, we have $n+1-k \not = k+1$.
Also, because $k \in \mathcal{I}(Y')$ and $n+2-k \not\in \mathcal{I}(Y')$, we have $k \not= n+2-k$, and hence $n+1-k \not= k-1$.
Thus, $p$ satisfies the orthogonality condition.
If $\#p = 1$, then $p = \{k\}$, and $\tilde{S}_p(Y') = S_k(Y') = Y$ by Lemma \ref{IandS}. 
If $\#p = 2$, then $k \not= n+1-k$ and $\{k,k+1\} \cap \{n+1-k, n+2-k\} = \emptyset$, which implies that $n+2-k \not\in \mathcal{I}(Y)$ by Remark \ref{IandS2}, and $n+1-k \not\in \mathcal{I}(Y)$ by Lemmas \ref{sign coincide} and \ref{IandS}.
Hence we have $\tilde{S}_p(Y') = S_{n+1-k}S_k(Y') = S_{n+1-k}(Y) = Y$. 
In both cases, we obtain $Y \in \tilde{\mathcal{F}}(Y_{m,n-m+1})$.
Assume that $n+2-k \in \mathcal{I}(Y')$.
Let $Y'' \in \mathcal{F}(Y_{m,n-m+1})$ be such that $\mathcal{I}(Y'') = \mathcal{I}(Y')\sqcup\{n+1-k\}\setminus\{n+2-k\}$; note that $\#Y'' = \#Y' -1$.
Because $n+2-(n+1-k) = k+1 \not\in \mathcal{I}(Y)$, we have $Y'' \in \mathcal{SS}(Y_{m,n-m+1})$.
By the induction hypothesis, it follows that $Y'' \in \tilde{\mathcal{F}}(Y_{m,n-m+1})$.
Because $\#Y'' = \#Y -2$, we have $\#p = 2$.
We see by Lemmas \ref{sign coincide} and \ref{IandS} that if $p$ satisfies the orthogonality condition, then $\tilde{S}_p(Y'') = S_kS_{n+1-k}(Y'') = S_k(Y') = Y$. 
If $p$ does not satisfy the orthogonality condition, then $n+1-k = k-1$.
Thus we obtain $k-1 \in \mathcal{I}(Y'')$, $k, k+1 \not\in \mathcal{I}(Y'')$, and hence $\tilde{S}_p(Y'') = S_{k-1}S_kS_{k-1}(Y'') = S_{k-1}Y = Y$. 
In both cases, we obtain $Y \in \tilde{\mathcal{F}}(Y_{m,n-m+1})$, as desired.
\end{proof}



\end{document}